\documentclass[10pt,reqno]{amsart}

\usepackage{color}

\usepackage{color}

\newtheorem{thm}{Theorem}[section]
\newtheorem{defn}[thm]{Definition}
\newtheorem{corollary}[thm]{Corollary}
\newtheorem{lem}[thm]{Lemma}

\newtheorem{assumption}[thm]{Assumption}

\theoremstyle{remark}
\newtheorem{remark}[thm]{Remark}

\def\qed{{\hfill $\Box$ \bigskip}}

\def\XXint#1#2#3{{\setbox0=\hbox{$#1{#2#3}{\int}$}
\vcenter{\hbox{$#2#3$}}\kern-.5\wd0}}

\newcommand\aint{-\hspace{-0.38cm}\int}

\newcommand\cbrk{\text{$]$\kern-.15em$]$}}
\newcommand\opar{\text{\,\raise.2ex\hbox{${\scriptstyle
|}$}\kern-.34em$($}}
\newcommand\cpar{\text{$)$\kern-.34em\raise.2ex\hbox{${\scriptstyle |}$}}\,}

\def\<{\langle}
\def\>{\rangle}

\newcommand\bL{\mathbb{L}}
\newcommand\bR{\mathbb{R}}
\newcommand\bC{\mathbb{C}}
\newcommand\bH{\mathbb{H}}

\newcommand\bN{\mathbb{N}}

\newcommand\fL{\mathbf{L}}
\newcommand\fR{\mathbf{R}}

\newcommand\cF{\mathcal{F}}
\newcommand\cG{\mathcal{G}}

\newcommand\cI{\mathcal{I}}

\newcommand\cL{\mathcal{L}}

\newcommand\cR{\mathcal{R}}

\newcommand\cM{\mathcal{M}}

\newcommand{\mysection}[1]{\section{#1}
\setcounter{equation}{0}}

\begin{document}

\title[$L_q(L_p)$-theory for Fractional Equations]{An $L_q(L_p)$-theory for the time fractional evolution equations with variable coefficients}

\author{Ildoo Kim}
\address{Department of Mathematics, Korea University, 1 Anam-dong, Sungbuk-gu, Seoul,
136-701, Republic of Korea} \email{waldoo@korea.ac.kr}

\author{Kyeong-Hun Kim}
\address{Department of Mathematics, Korea University, 1 Anam-dong,
Sungbuk-gu, Seoul, 136-701, Republic of Korea}
\email{kyeonghun@korea.ac.kr}
\thanks{This work was supported by Samsung Science  and Technology Foundation under Project Number SSTF-BA1401-02}

\author{Sungbin Lim}
\address{Department of Mathematics, Korea University, 1 Anam-dong, Sungbuk-gu, Seoul,
136-701, Republic of Korea} \email{sungbin@korea.ac.kr}



\begin{abstract}
We introduce an  $L_q(L_p)$-theory for the quasi-linear fractional  equations of the type
$$
\partial^{\alpha}_tu(t,x)=a^{ij}(t,x)u_{x^ix^j}(t,x)+f(t,x,u), \quad t>0, \,x\in \fR^d.
$$
 Here, $\alpha\in (0,2)$, $p,q>1$, and $\partial^{\alpha}_t$  is the Caupto fractional derivative of order $\alpha$.  Uniqueness, existence, and $L_q(L_p)$-estimates of solutions are obtained.  The leading coefficients $a^{ij}(t,x)$ are assumed to be piecewise continuous in $t$ and uniformly continuous in $x$.
  In particular $a^{ij}(t,x)$ are allowed to be discontinuous with respect to the time variable.
  Our approach is based on classical tools in PDE theories such as the Marcinkiewicz interpolation theorem, the Calderon-Zygmund theorem, and perturbation arguments.
\end{abstract}

\maketitle

\mysection{Introduction}

Fractional calculus has been used in numerous areas including mathematical modeling \cite{MS, V}, control engineering \cite{caponetto2010fractional, podlubny1999fractional},
electromagnetism \cite{engheia1997role, tarasov2006electromagnetic}, polymer science \cite{bagley1983theoretical, metzler1995relaxation}, hydrology \cite{benson2000application, SBMW}, biophysics \cite{glockle1995fractional, langlands2009fractional}, and even finance \cite{raberto2002waiting, scalas2000fractional}. See also \cite{hilfer2000applications, ortigueira2011fractional, sabatier2007advances, zaslavsky2002chaos} and references therein.
The classical heat equation $\partial_t u=\Delta u$ describes the heat propagation in homogeneous mediums.
The time-fractional diffusion equation $\partial^\alpha_t u=\Delta  u$, $\alpha\in (0,1)$,  can be used
to model the anomalous diffusion  exhibiting  subdiffusive behavior,
due to particle sticking and trapping phenomena (see \cite{metzler1999anomalous, metzler2004restaurant}). The fractional wave equation $\partial_t u=\Delta u$, $\alpha\in (1,2)$  governs the propagation of mechanical diffusive waves in viscoelastic media (see \cite{mainardi1995fractional}).
The fractional differential equations have an another important issue in the probability theory related to non-Markovian diffusion processes with a memory (see \cite{metzler2000boundary, MK}).

The main goal of this article is to present an $L_q(L_p)$-theory for the quasi-linear fractional evolution equation
\begin{equation}
                       \label{main eqn 1}
\partial^{\alpha}_tu(t,x)=a^{ij}(t,x)u_{x^ix^j}(t,x)+b^i(t,x)u_{x^i}(t,x)+c(t,x)u(t,x)+f(t,x,u)
\end{equation}
given for $t > 0$ and $x\in \fR^d$. Here $\alpha\in (0,2), p,q>1$, and $\partial^{\alpha}_t$ denotes the Caputo fractional derivative (see \eqref{eqn defn caupto}).
The indices $i$ and $j$ move from $1$ to $d$, and the summation convention with respect to the repeated indices is assumed throughout the article.
It is assumed that the leading coefficients $a^{ij}(t,x)$ are piecewise continuous in $t$ and uniformly continuous in $x$, and the lower order coefficients $b^i$ and $c$ are only bounded measurable functions. We prove that under a mild condition on the nonlinear term $f(t,x,u)$ there exists a unique solution $u$ to \eqref{main eqn 1} and the $L_q(L_p)$-norms of the derivatives $D^{\beta}_x u$, $|\beta|\leq 2$, are controlled by the $L_q(L_p)$-norm of $f(t,x,0)$.

We remark that there are a few other types of fractional derivatives such as
 Riemann-Liouville, Marchaud,  and Gr\"{u}nwald-Letnikov fractional derivatives.  These three
fractional derivatives coincide with the Caputo fractional derivative in our solution space $\mathbb{H}_{q,p,0}^{\alpha,n}(T)$ (see \cite{SKM,Po} for the proof).

Here is a brief survey of closely related works.
In \cite{clement1992global, Pr1991} an $L_q(L_p)$-theory for the parabolic Volterra equations of the type
\begin{equation}
         \label{eqn pruss}
\frac{\partial}{\partial t}\left(c_0u+\int^t_{-\infty}k_1(t-s)u(s,x)ds\right)=\Delta u+f(t,x,u), \quad t\in \fR, x\in \fR^d
\end{equation}
is obtained under the conditions  $k_1(t)\geq ct^{-\alpha}$ for small $t$,  $c_0\geq 0$, $\alpha\in (0,1)$, and
\begin{equation}
                   \label{eqn pruss cond}
\frac{2}{\alpha q}+\frac{d}{p}<1.
\end{equation}
The results of \cite{clement1992global, Pr1991} also cover the case $c_0>0$, however it is obtained only for the case $a^{ij}(t,x)=\delta^{ij}$ with  the restrictions  $\alpha\in (0,1)$ and \eqref{eqn pruss cond}.  If $p=q$, an $L_p$-theory of type \eqref{main eqn 1} with the variable coefficients $a^{ij}(t,x)$ is presented in \cite{zacher2005maximal} under the condition that $a^{ij}$ are uniformly continuous  in $(t,x)$ and  $\lim_{|x|\to \infty}a^{ij}(t,x)$ exists. In \cite{Za} an $L_{2}$-theory is obtained for the divergence type equations with general measurable coefficients. Also an eigenfunction expansion method is introduced in \cite{SY} to obtain $L_2$-estimates of solutions of divergence type equations with $C^1$-coefficients.

For other approaches to the  equations with fractional time derivatives, we  refer  to \cite{kunstmann2004maximal} for the semigroup approach,  to \cite{da1985existence, Pr} for $C^{\delta}([0,T], X)$-type theory, where $X$ is an appropriate Banach space, and to \cite {clement2004quasilinear} for $BUC_{1-\beta}([0,T],X)$-type estimate, where $\|u\|_{BUC_{1-\beta}([0,T],X)}=\sup_{t\in(0,T]}t^{1-\beta}\|u(t)\|_{X}$.

Our result substantially generalizes above mentioned results in the sense that we do not impose any algebraic conditions on $\alpha$, $p$, and $q$. The conditions \eqref{eqn pruss cond} and $\alpha\in (0,1)$ are used in  \cite{clement1992global, Pr1991}, and the restrictions $p=q$ and $\alpha\not\in \{\frac{2}{2p-1}, \frac{2}{p-1}-1,\frac{1}{p}, \frac{3}{2p-1}\}$  are assumed in \cite{zacher2005maximal}. More importantly, in this article the condition on the leading coefficients $a^{ij}(t,x)$ is considerably weakened. In particular, $a^{ij}(t,x)$ depend on both $t$ and $x$ and can be  discontinuous in $t$. Recall that if  $p>2$ then among above  articles only \cite{zacher2005maximal} considers the coefficients depending also on $t$, but the condition $p=q$ and the continuity  of $a^{ij}$ with respect to $(t,x)$ are assumed in \cite{zacher2005maximal}.

Another significance of this article is the method we use. The results of \cite{clement1992global, Pr1991,zacher2005maximal} are  operator theoretic,  \cite{kunstmann2004maximal} is based on $H^{\infty}$-functional calculus, and the method of \cite{Za,SY} works well only in the Hilbert-space framework. Our approach is purely analytic and is based on the classical tools in PDE theories including the Marcinkiewicz interpolation theorem and the Calderon-Zygmund theorem. We obtain the mean oscillation (or BMO estimate) of solutions and then apply the Marcinkiewicz interpolation theorem to obtain $L_p$-estimates of solutions. To go from $L_p$-theory to  $L_q(L_p)$-theory we show that the kernel appeared in the representation of solutions for equations with constant coefficients satisfies the conditions needed for the Calderon-Zygmund theorem.  Perturbation and fixed point arguments are used to handle the variable coefficients and the nonlinear term respectively.

The article is organized as follows.  Some properties of the fractional derivatives and our main result, Theorem \ref{main theorem}, are presented in Section 2. The representation of solutions to a model equation and an $L_2$-estimate of solutions are given in Section 3, and BMO and an $L_q(L_p)$-estimate of solutions to a model equation are obtained in Section 4. The proof of Theorem \ref{main theorem} is given in Section 5, and sharp estimates of kernels related to the representation of solutions are obtained in Section 6.

We finish the introduction with  some notation used in this article. As usual $\bN=\{1,2,\cdots\}$, $\fR^{d}$ stands for the Euclidean space of points
$x=(x^{1},...,x^{d})$,  $B_r(x) := \{ y\in \fR^d : |x-y| < r\}$,  and
$B_r :=B_r(0)$.
 For  multi-indices $\gamma=(\gamma_{1},...,\gamma_{d})$,
$\gamma_{i}\in\{0,1,2,...\}$, $x \in \fR^d$, and  functions $u(x)$ we set
$$
 u_{x^{i}}=\frac{\partial u}{\partial x^{i}}=D_{i}u,\quad \quad
D^{\gamma}_xu=D_{1}^{\gamma_{1}}\cdot...\cdot D^{\gamma_{d}}_{d}u,
$$
$$
x^\gamma = (x^1)^{\gamma_1} (x^2)^{\gamma_2} \cdots (x^d)^{\gamma_d},\quad \quad
|\gamma|=\gamma_{1}+\cdots+\gamma_{d}.
$$
We also use $D^m_x$ to denote a partial derivative of order $m$  with respect to $x$.
For an open set $\Omega \subset \fR^d$
by $C_c^\infty(\Omega)$  we denote the set of infinitely differentiable  functions with compact support in $U$.
For a Banach space $F$ and $p>1$ by $L_p(U,F)$ we denote the set of $F$-valued Lebesgue-measurable functions $u$ on $\Omega$ satisfying
$$
\|u\|_{L_{p}(\Omega,F)}=\left(\int_{\Omega}\|u(x)\|_{F}^{p}dx\right)^{1/p}<\infty.
$$
We write $f\in L_{p,loc}(U,F)$ if  $\zeta f\in L_p(U,F)$ for any real-valued $\zeta\in C_c^\infty(U)$.
Also $L_p(\Omega)=L_p(\Omega, \fR)$ and $L_p=L_p(\fR^d)$.
We use  ``$:=$" to denote a definition.
By $\cF$ and $\cF^{-1}$ we denote the $d$-dimensional Fourier transform and the inverse Fourier transform respectively, i.e.
$$
\cF(f)(\xi) := \frac{1}{(2\pi)^{d/2}} \int_{\fR^{d}} e^{-i x \cdot \xi} f(x) dx, \quad
\cF^{-1}(f)(x) := \frac{1}{(2\pi)^{d/2}}\int_{\fR^{d}} e^{ i\xi \cdot x} f(\xi) d\xi.
$$
For a Lebesgue set $A\subset \fR^d$, we use $|A|$ to denote its Lebesgue
measure and by $1_A(x)$ we denote  the indicator of $A$.
For a complex number $z$, $\Re[z]$ is the real part of $z$. Finally if we write $N=N(a,b,\ldots)$, this means that the constant $N$ depends only on $a,b,\ldots$.

\mysection{Main results}

We fix $T \in (0,\infty)$ throughout the article.  For $\alpha>0$  denote
$$
k_{\alpha}(t):=t^{\alpha-1}\Gamma(\alpha)^{-1}, \quad t>0,
$$
where $\Gamma(\alpha):=\int^{\infty}_0 t^{\alpha-1}e^{-t}dt$.
For functions $\varphi\in L_1((0,T))$  the Riemann-Liouville fractional integral of the order $\alpha>0$ is defined as
$$
I^{\alpha} \varphi (t)=k_{\alpha}*\varphi (t)=\frac{1}{\Gamma(\alpha)}\int^t_0 (t-s)^{\alpha-1}\varphi(s)ds.
$$
It is easy to check that
$$
I^{\alpha+\beta} \varphi (t) =I^\alpha I^\beta \varphi (t) \quad \quad \forall \alpha, \beta >0.
$$
Also,  by Jensen's inequality,
 for any $p\in [1,\infty]$,
\begin{equation}
                        \label{eqn 7.03.1}
\|I^{\alpha}\varphi\|_{L_p((0,T))}\leq N(T,\alpha) \|\varphi\|_{L_p((0,T))}.
\end{equation}
It is also known (see e.g. \cite{SKM}) that $I^{\alpha}: B^{\lambda}\to C^{\lambda+\alpha}$ is a bounded operator if $\lambda\geq 0$ and $\lambda+\alpha<1$,
where $B^0=L_{\infty}$ and  $B^{\lambda}=C^{\lambda}$ if $\lambda>0$.

Let $k\in \bN$,  $k-1\leq \alpha<k$, and $f^{(k-1)}(t)$ be absolutely continuous, where $f^{(k-1)}(t)$ denotes the $(k-1)$-th derivative  of function $f$.
Then the Caputo fractional derivative of order $\alpha>0$ is defined as
\begin{align}
                       \label{eqn defn caupto-0}
\partial^{\alpha}_t f(t)
&=\frac{1}{\Gamma(k-\alpha)}\int^t_0 (t-s)^{k-\alpha-1}f^{(n)}(s)ds\\
&=\frac{1}{\Gamma(k-\alpha)}\frac{d}{dt}\int^t_0 (t-s)^{k-\alpha-1}\left[f^{(k-1)}(s)-f^{(k-1)}(0)\right] ds.
\label{eqn defn caupto}
\end{align}
Note that  \eqref{eqn defn caupto} (``Kochubei extension") is defined for a broader class of functions.

Let $q\geq 1$. For functions $f\in C^k([0,T])$, we denote
    $$
    \|f\|_{H^{\alpha}_q(T)}=\left(\int^T_0 |f|^q dt\right)^{1/q}+ \left(\int^T_0 |\partial^{\alpha}_tf|^q dt \right)^{1/q}.
    $$
The following lemma shows that it is irreverent whether one uses \eqref{eqn defn caupto-0} or \eqref{eqn defn caupto} as the definition of $\partial^{\alpha}_t$ for functions in $H^{\alpha}_q(T)$.
\begin{lem}
   The  closures  $H^{\alpha}_q(T)$ and
        $\tilde{H}^{\alpha}_q(T)$ of $C^k([0,T])$ in the space $L_q((0,T))$ with respect to norms  $\|\cdot\|_{H^{\alpha}_q}$ related to
        \eqref{eqn defn caupto-0}  and \eqref{eqn defn caupto} respectively coincide.
 \end{lem}
\begin{proof}
This is obvious because  \eqref{eqn defn caupto-0}  and \eqref{eqn defn caupto} are equal for functions $f\in C^k([0,T])$.
\end{proof}

Next we introduce an another fractional derivative. Let $D^{\alpha}_t$ denote the Riemann-Liouville fractional derivative of order $\alpha$ which is defined as
\begin{equation}
          \label{Riemann-Liouville}
D^{\alpha}_t\varphi (t)=\frac{d}{dt} (I^{1-\alpha}\varphi)(t), \quad \alpha\in (0,1).
\end{equation}
It is obvious that
\begin{align}
          \label{caputo}
\partial^{\alpha}_t\varphi (t)=D^{\alpha}_t(\varphi-\varphi(0))
= D^{\alpha}_t\varphi (t)-\frac{\varphi(0)}{t^{\alpha}\Gamma(1-\alpha)}, \quad \alpha\in (0,1).
\end{align}
It is easy to check for any $\varphi\in L_1((0,T))$,
\begin{align}
          \label{e:DI}
D^{\alpha}_tI^{\alpha}\varphi =\varphi, \quad \alpha\in (0,1).
\end{align}
Similarly, the equality
\begin{equation}
                 \label{e:ID}
I^{\alpha}D_t^{\alpha}\varphi=\varphi, \quad \alpha\in (0,1)
\end{equation}
also holds if $I^{1-\alpha}\varphi$ is absolutely continuous and $I^{1-\alpha}\varphi(0)=0$.

\begin{defn}
Let $k-1\leq\alpha<k$ and  $f\in H^{\alpha}_q(T)$.  We write $f(0)=0$ if there exists a sequence $f_{n}\in C^k([0,T])$ such that $f_n(0)=0$ and $f_n \to f$ in $H^{\alpha}_q(T)$.
 Similarly if $\alpha>1$ we write $f'(0)=0$ if $f'_n(0)=0$ for all $n$.
\end{defn}

The following lemma gives sufficient and necessary conditions for  $f\in H^{\alpha}_q(T)$ and $f(0)=0$ (or $f'(0)=0$).
\begin{lem}
(i) Let $\alpha\in (0,1)$ and $q>1$. Then $f\in H^{\alpha}_q(T)$ and $f(0)=0$ if and only if
$f\in L_q((0,T))$, $I^{1-\alpha}f\in H^1_q(T)$, $I^{1-\alpha}f(t)$ is continuous, and $I^{1-\alpha}f(0)=0$.

(ii) Let $\alpha \in (1,2)$ and $q>1$. Then $f\in H^{\alpha}_q(T)$ and $f'(0)=0$ if and only if
$f\in H^1_q(T)$,  $I^{2-\alpha}f'\in H^1_q(T)$,  $I^{2-\alpha}f'(t)$ is continuous, and $I^{2-\alpha}f'(0)=0$.

\end{lem}

\begin{proof}
(i) Suppose $f\in H^{\alpha}_q(T)$ and $f(0)=0$.
Take a sequence $f_n \in C^1([0,T])$ satisfying $f_n(0)=0$ and $f^n \to f$ in $H^{\alpha}_q(T)$.
Then  we have $f_n \to f$ and $\frac{d}{dt} (I^{1-\alpha}f_n)\to \partial^{\alpha}_t f$  in $L_q((0,T))$.
It follows that $I^{1-\alpha}f^n$ converges to  $I^{1-\alpha}f$ in the space $H^1_q((0,T))$. It also follows that
$$
I^{1-\alpha}f(t)=\int^t_0 \partial^{\alpha}_tf(s) ds, \quad t \leq T \quad  (a.e.).
$$
Since $1-1/q>0$, by the Sobolev embedding theorem $I^{1-\alpha}f(t)$ is continuous in $t$ and thus the above equality holds for all $t$,
which implies $I^{1-\alpha}f(0)=0$.

Next we prove the ``if" part  of (i). By the assumption, we can choose a function $g\in L_q((0,T))$ so that
$$
I^{1-\alpha}f(t)=\int^t_0 g(s)ds, \quad \forall \, t\leq T.
$$
Take a sequence of functions $g_n \in C^1([0,T])$ which converges to $g$ in $L_q((0,T))$. Define
$f_n=I^{\alpha}g_n$. Then $f_n\in C^1([0,T])$, $f_n(0)=0$, and
$$
f_n \to I^\alpha g=f \quad \text{in} \quad L_q((0,T)).
$$
Also
$$
\partial^{\alpha}_tf_n=\partial^{\alpha}_tI^{\alpha}g_n=g_n \to \partial^{\alpha}_t f \quad \text{in} \quad L_q((0,T)).
$$

(ii) The proof is very similar to (i). We only explain how one can choose a sequence to prove  the ``if" part.

By the assumption $f\in H^1_q(T)$ and $q>1$, we may assume $f \in C([0,T])$.
Take $g\in L_q((0,T))$ so that $I^{2-\alpha}f'=\int^t_0 g(s) ds$.
We choose   $g_n \in C^2([0,T])$ which converges to $g$ in $L_q((0,T))$.  Define
$f_n(t)=I^{\alpha}g_n(t)+f(0)$. Then $f'_n(0)=I^{\alpha-1}g_n (0)=0$, $f_n \to f $ in $L_q((0,T))$, and $\partial^\alpha_t f_n$ is a Cauchy sequence in $L_q((0,T))$. Thus $f\in H^{\alpha}_p(T)$ and $f'(0)=0$.

The lemma is proved.
\end{proof}

Next we introduce our solution space $\bH^{\alpha,k}_{q,p}(T)$ and related notation.
Roughly speaking, we write $u\in \bH^{\alpha,k}_{q,p}(T)$ if and only if
$$
u, \,\, \partial^{\alpha}_tu,  \,\, D^k_x u \,\,\in \,L_q((0,T),L_p).
$$
For $p,q>1$ and $k=0,1,2,\cdots$, we denote
$$
H^k_p=H^k_p(\fR^d)=\{u \in L_p(\fR^d) : D^{\gamma}_x u\in L_p(\fR^d), |\gamma|\leq k\},
$$
$$
\bH^{0,k}_{q,p}(T)=L_q((0,T),H^k_p), \quad \bL_{q,p}(T):=\bH^{0,0}_{q,p}(T),
$$
where $D^\gamma_x$ are derivatives in the distributional sense.
Thus $u\in \bH^{0,k}_{q,p}(T)$ if and only if $u(t,\cdot)$ is $H^k_p$-valued measurable function satisfying
$$
\|u\|_{\bH^{0,k}_{q,p}(T)}:=\left[\int^T_0 \|u\|^q_{H^k_p}ds\right]^{1/q} <\infty.
$$
We extend the real-valued time fractional Sobolev space to $L_p(\fR^d)$-valued one. In other words,
we consider the completion of $C^2([0,T] \times \fR^d) \cap \bL_{q,p}(T)$ with respect to norm
$$
\|\cdot\|_{\bL_{q,p}(T)} + \|\partial_t^\alpha \cdot\|_{\bL_{q,p}(T)}
$$
in $\bL_{q,p}(T)$.
\begin{defn}
For $\alpha \in (0,2)$ we say $u \in \bH^{\alpha,0}_{q,p}(T)$ if and only if there exists a sequence $u_n \in C^2([0,T] \times \fR^d) \cap \bL_{q,p}(T)$ so that
$\sup_n \|\partial_t^\alpha u_n\|_{\bL_{q,p}(T)} < \infty$, and
$$
\|u-u_n\|_{\bL_{q,p}(T)} \to 0 \quad \text{and} \quad \|\partial_t^\alpha u_n-\partial_t^\alpha u_m\|_{\bL_{q,p}(T)} \to 0
$$
as $n$ and $m$ go to infinity. We call this sequence $u_n$ a defining sequence of $u$.
For $u \in \bH^{\alpha,0}_{q,p}(T)$, we define
$$
\partial_t^\alpha u = \lim_{n \to \infty}\partial_t^\alpha u_n \quad \text{in} \quad \bL_{q,p}(T),
$$
where $u_n$ is a defining sequence of $u$.
Obviously $\bH^{\alpha,0}_{q,p}(T)$ is a Banach space
with the norm
$$
\|u\|_{\bH^{\alpha,0}_{q,p}(T)}=\|u\|_{\bL_{q,p}(T)} +\|\partial_t^\alpha u\|_{\bL_{q,p}(T)}.
$$
\end{defn}
\begin{defn}
                    \label{sol spa defn}
For $u \in \bH^{\alpha,0}_{q,p}(T)$, we say that $u(0, x)= 0$ if any only if there exists a defining sequence $u_n$ such that
$$
u_n(0,x)= 0  \quad \quad \forall x \in \fR^d, \quad \forall n \in \bN.
$$
Similarly we say that $u(0, \cdot)= 0$ and $\frac{\partial}{\partial t}u(0, \cdot)= 0$ if any only if there exists a defining sequence $u_n$ such that
$$
u_n(0,x)= 0  \quad \text{and} \quad \frac{\partial}{\partial t}u_n(0,x)=0 \quad \quad \forall x \in \fR^d, \quad \forall n \in \bN.
$$
\end{defn}

Let
$$
\bH^{\alpha,k}_{q,p}(T) := \bH^{\alpha,0}_{q,p}(T) \cap \bH^{0,k}_{q,p}(T)
$$
and
$\bH^{\alpha,k}_{q,p,0}(T)$ be the subspace of $\bH^{\alpha,k}_{q,p}(T)$ such that
$$
u(0,\cdot) =0 \quad \quad \text{if} \quad \alpha \in (0,1]
$$
and
$$
u(0,\cdot) =0 \quad \text{and} \quad \frac{\partial}{\partial t}u(0,\cdot)=0 \quad \quad\text{if} \quad \alpha \in (1,2).
$$

\begin{thm}
               \label{lem 9.21.11}
%
%
%
(i) The space $\bH^{\alpha,k}_{q,p}(T)$ is a Banach space with the  norm
$$
\|u\|_{\bH^{\alpha,k}_{q,p}(T)} :=\|u\|_{\bH^{0,k}_{q,p}(T)}+\|u\|_{\bH^{\alpha,0}_{q,p}(T)}.
$$

\noindent
(ii)  The space $\bH^{\alpha,k}_{q,p,0}(T)$ is a closed subspace of $\bH^{\alpha,k}_{q,p}(T)$.

\noindent
(iii)  $C_c^\infty(\fR^{d+1}_+)$ is dense in $\bH^{\alpha,k}_{q,p,0}(T)$.

\noindent
(iv) For  any $u\in \bH^{\alpha,2}_{q,p,0}(T)$,
\begin{equation}
              \label{eqn 9.21}
\|u(t)\|_{L_p}\leq N(\alpha)\int^t_0 (t-s)^{\alpha-1}\|\partial^{\alpha}_t u(s)\|_{L_p}ds, \quad \quad t\leq T \quad (a.e.).
\end{equation}
Consequently,
$$
\|u\|^q_{\bL_{q,p}(t)}\leq N(q,\alpha,T)\int^t_0 \int^s_0 (s-r)^{\alpha-1}\|\partial^{\alpha}_tu(r)\|^q_{L_p}drds, \quad \quad \forall \, t\leq T.
$$
\end{thm}
\begin{proof}
(i) This is obvious because both $\bH^{\alpha,0}_{q,p}(T)$ and $\bH^{0,k}_{q,p}(T)$ are Banach spaces.

(ii)
Suppose $u_n\in \bH^{\alpha,k}_{q,p,0}(T)$ and  $u \in \bH^{\alpha,k}_{q,p}(T)$ so that $u_n \to u$ in $\bH^{\alpha,k}_{q,p}(T)$.
Since $u_n \in \bH^{\alpha,k}_{q,p,0}(T)$ for each $n$, we can find a $v_n \in C^2([0,T]\times \fR^d) \cap \bL_{q,p}(T)$ so that
$$
\|u_n-v_n\|_{\bL_{q,p}(T)} < \frac{1}{n} \quad \text{and} \quad \|\partial_t^\alpha u_n-\partial_t^\alpha v_n\|_{\bL_{q,p}(T)} < \frac{1}{n},
$$
$$
v_n(0,x)= 0  \quad \quad \forall x \in \fR^d, \quad  \text{if} \quad \alpha \in (0,1],
$$
and
$$
v_n(0,x)= 0  \quad \text{and} \quad \frac{\partial}{\partial t}v_n(0,x)=0 \quad \quad \forall x \in \fR^d, \quad \text{if} \quad \alpha \in (1,2).
$$
Therefore $u\in \bH^{\alpha,k}_{q,p,0}(T)$ because obviously
$$
\|v_n-u\|_{\bH^{\alpha,0}_{q,p}(T)} \to 0 \quad \text{as} \quad  n \to \infty.
$$
This certainly proves (ii).

(iii) We
take nonnegative smooth functions $\eta_1 \in C_c^\infty((1,2))$, $\eta_2 \in C_c^\infty(\fR^d)$, and $\eta_3 \in C_c^\infty(\fR^d)$ so that
$$
\int_0^\infty \eta_1(t)~dt=1, \quad \int_{\fR^d} \eta_2(x)~dx=1, \quad \text{and} \quad \eta_3(x)=1\quad \text{if}\quad |x|\leq 1.
$$
For $\varepsilon_1, \varepsilon_2, \varepsilon_3 >0$, we define
$$
\eta_{1,\varepsilon_1}(t)= \varepsilon_1^{-1} \eta_1(t/\varepsilon_1), \quad \eta_{2,\varepsilon}(x)= \varepsilon_2^{-d} \eta_2(x/\varepsilon_2),
$$
$$
u^{\varepsilon_1}(t,x)=\int_0^\infty u(s,x) \eta_{1,\varepsilon_1}(t-s)ds,
$$
$$
u^{\varepsilon_1 ,\varepsilon_2}(t,x)=\int_{\fR^d}\int_0^\infty u(s,y) \eta_{1,\varepsilon_1}(t-s)\eta_{2,\varepsilon_2}(x-y)dsdy,
$$
and
$$
u^{\varepsilon_1 ,\varepsilon_2, \varepsilon_3}(t,x)
=\eta(t)\eta_{3}(\varepsilon_3 x)\int_{\fR^d}\int_0^\infty u(s,y) \eta_{1,\varepsilon_1}(t-s)\eta_{2,\varepsilon_2}(x-y)dsdy,
$$
where $\eta  \in C^\infty([0,\infty))$ such that $\eta(t)=1$ for all $t \leq T$ and vanishes for all large $t$.
Due to the condition $\eta_1 \in C_c^\infty((1,2))$, it holds that
$$
u^{\varepsilon_1 ,\varepsilon_2, \varepsilon_3}(t,x) =0\quad \quad \forall t <\varepsilon_1, \quad \forall x \in \fR^d.
$$
We can easily check that for any $u \in \bH^{\alpha,k}_{q,p,0}(T)$
$$
\partial_t^\alpha u^{\varepsilon_1} (t) = (\partial_t^\alpha u)^{\varepsilon_1}(t).
$$
Hence for any given $\varepsilon>0$ we have
\begin{align*}
&\|u-u^{\varepsilon_1,\varepsilon_2,\varepsilon_3}\|_{\bH^{\alpha,k}_{q,p}(T)}  \\
&\leq \|u-u^{\varepsilon_1}\|_{\bH^{\alpha,k}_{q,p}(T)}
+\|u^{\varepsilon_1}-u^{\varepsilon_1,\varepsilon_2}\|_{\bH^{\alpha,k}_{q,p}(T)}
+\|u^{\varepsilon_1,\varepsilon_2}-u^{\varepsilon_1,\varepsilon_2,\varepsilon_3}\|_{\bH^{\alpha,k}_{q,p}(T)}  \leq \varepsilon
\end{align*}
if $\varepsilon_1$, $\varepsilon_2$, and $\varepsilon_3$ are small enough.
Therefore (iii) is proved.

(iv)
Due to (iii), it is enough to prove \eqref{eqn 9.21} only for  $u \in C_c^\infty(\fR^{d+1})$.
Denote $f:=\partial^{\alpha}_tu$. One can easily check
 $$
 u(t)=\int^t_0k_{\alpha}(t-s)f(s)ds, \quad  \forall t\leq T
 $$
in the space $L_p$, which clearly implies \eqref{eqn 9.21} due to the generalized Minkowski inequality.
The theorem is proved.
\end{proof}

\begin{assumption}
                                \label{ass 9.19}
    Let $f(u)=f(t,x,u)$ and $f_0=f(t,x,0)$.

\noindent
(i)  There exist   $0=T_0 <T_1 <\cdots < T_{\ell}=T$ and functions $a^{ij}_k(t,x)$ such that
$$
a^{ij}(t,x) = \sum_{k=1}^{\ell} a^{ij}_k(t,x) I_{(T_{k-1},T_k]}(t), \quad (a.e.).
$$

\noindent
(ii) There exist constants $\delta, K>0$ so that for any $k$, $t$, and $x$

\begin{equation}
                         \label{eqn elliptic}
\delta |\xi|^2 \leq a_k^{ij}(t,x)\xi^i\xi^j \leq K|\xi|^2, \quad \forall \xi\in \fR^d,
\end{equation}
and
$$
|a_k^{ij}(t,x)|+|b^i(t,x)|+|c(t,x)|\leq K.
$$

\noindent
(iii) The coefficients $a_k^{ij}$ are uniformly continuous on $(t_{k-1},t_k] \times \fR^d$ for all $k=1,\ldots,\ell$ and $i,j =1,\ldots,d$.

\noindent
(iv) $f_0\in \bL_{q,p}(T)$ and $f(u)$ satisfies the following continuity property: for any $\varepsilon>0$, there exists a constant $K_{\varepsilon}>0$ such that
\begin{equation}
     \label{same}
\|f(t,x,u)-f(t,x,v)\|_{L_p}\leq \varepsilon \|u-v\|_{H^2_p}+K_{\varepsilon}\|u-v\|_{L_p},
\end{equation}
for any $(t,x)$ and  $u,v\in H^2_p$.
\end{assumption}

If $p\neq q$ then we need an additional condition (see the comment below \eqref{equiv rel}  for the reason).
\begin{assumption}
                 \label{ass 9.21}
                 If $p\neq q$ then
                $\lim_{|x|\to \infty}a^{ij}(t,x)$ exists uniformly in $t\in (0,T)$.
\end{assumption}

Here is the main result of this article. The proof  will be given in Section \ref{pf main thm}.

\begin{thm}
                        \label{main theorem}
Let $p,q>1$ and Assumptions \ref{ass 9.19} and \ref{ass 9.21} hold. Then the equation
\begin{equation}
               \label{main eqn}
 \partial^{\alpha}_tu=a^{ij}u_{x^ix^j}+b^iu_{x^i}+cu+f(u), \quad t>0
 \end{equation}
 admits a unique solution $u$ in the class $\bH^{\alpha,2}_{q,p,0}(T)$, and for this solution it holds that
 \begin{equation}
                             \label{main estimate}
 \|u\|_{\bH^{\alpha,2}_{q,p}(T)}\leq N_0\|f_0\|_{\bL_{q,p}(T)},
 \end{equation}
 where $N_0$ depends only on $d,p,q,\delta,K,K_{\varepsilon}, T, \ell$, and the modulus of continuity of $a^{ij}_k$.
\end{thm}

\begin{remark}
(i) Due to the definition of our solution space $\bH^{\alpha,2}_{q,p,0}(T)$, the zero initial condition is given to equation \eqref{main eqn}, that is
 $$
u(0,x) =0  \quad \text{and additionally}\quad \frac{\partial u}{\partial t}(0,x)=0 \,\,\, \text{if}\,\, \alpha>1.
$$
%

(ii) Some examples of $f(u)$ satisfying \eqref{same} can be found e.g. in \cite{Krylov1999}. For instance, let
  $\kappa:=2-d/p>0$, $h=h(x)\in L_p$, and
  $$
  f(x,u):=h(x) \sup_x |u|.
  $$
  Then by the Sobolev embedding theorem,
  \begin{align}
                 \nonumber
  \|f(u)-f(v)\|_{L_p} &\leq \|h\|_{L_p} \sup_x |u-v|\leq N\|u-v\|_{H^{\kappa}_p}
  \leq \varepsilon  \|u-v\|_{H^2_p}+K \|u-v\|_{L_p}.
  \end{align}
Similarly one can show that $f(t,x,u):=a(t,x)(-\Delta)^{\delta}u$ also satisfies \eqref{same} if $a(t,x)$ is bounded and  $\delta\in (0,1)$. \end{remark}

\mysection{Some Preliminaries}
                        \label{section 2}
In this section we introduce some estimates of kernels related to the representation of a solution, and we also present an $L_2$-estimate of a solution.

 The Mittag-Leffler function $E_{\alpha}(z)$ is  defined as
$$
E_{\alpha}(z)=\sum_{k=0}^{\infty}\frac{z^k}{\Gamma(\alpha k+1)}, \quad z \in \bC, \,\alpha>0.
$$
The series converges for any $z\in \bC$, and $E_{\alpha}(z)$ is an entire function. Using
 $$
 \partial^{\alpha}_t t^{\beta}=\frac{\Gamma(\beta+1)}{\Gamma(\beta+1-\alpha)}t^{\beta-\alpha}, \quad \beta \geq \alpha
 $$
one can easily check that   for any constant $\lambda$,
\begin{equation*}
\varphi(t):=E_{\alpha}(\lambda t^{\alpha})
\end{equation*}
satisfies $\varphi(0)=1$ (also $\varphi'(0)=0$ if $\alpha>1$) and
\begin{equation*}
\partial^{\alpha}_t \varphi=\lambda \varphi, \quad t>0.
\end{equation*}
By taking the Fourier transform to the equation
$$
\partial^{\alpha}_tu=\Delta u, \quad t>0, \quad u(0)=h, \quad (\text{and}\, u'(0)=0 \,\, \text{if}\, \alpha>1)
$$
one can formally get $\cF(u)(t,\xi)=E_{\alpha}(-t^{\alpha}|\xi|^2)\tilde{h}$. Thus it is naturally needed to find an integrable function $p(t,x)$ satisfying $\cF(p(t,\cdot))(\xi)=E_{\alpha}(-t^{\alpha}|\xi|^2)$. It is known that   (see e.g. (12) in  \cite[Theorem 1.3-4]{Dj})

\begin{equation}
          \label{mittag}
E_{\alpha}(-t)\sim \frac{1}{1+|t|}, \quad \quad t>0.
\end{equation}
Therefore, $E_{\alpha}(-t^{\alpha}|\xi|^2)$  is integrable  only if $d=1$. If $d\geq 2$ then $\cF^{-1} (E_{\alpha}(-t^{\alpha}|\xi|^{2}))$ might be
 understood as an improper integral. However, in this article we do not consider $\cF^{-1} (E_{\alpha}(-t^{\alpha}|\xi|^{2}))$.

 \begin{lem}
                  \label{p exists}
(i) Let $d\geq 1$ and $\alpha\in (0,2)$. Then
there exists a function $p(t,x)$ such that $p(t,\cdot)$ is integrable in $\fR^d$ and
$$
\cF( p(t,\cdot))(\xi)=E_{\alpha}(-t^{\alpha}|\xi|^2).
$$
(ii)  Let  $m,n=0,1,2,\cdots$ and denote $R=t^{-\alpha}|x|^2$.
Then there exist constants $C$ and $\sigma$ depending only on $m,n,d$, and $\alpha$ so that
 if $R\geq 1$
\begin{equation}
                                                                               \label{p-1}
|\partial_{t}^{n}D_{x}^{m}p(t,x)|\leq N t^{\frac{-\alpha(d+m)}{2}-n}\exp\{-\sigma t^{-\frac{\alpha}{2-\alpha}}|x|^{\frac{2}{2-\alpha}}\},
\end{equation}
and  if $R\leq 1$
\begin{equation}
         \label{p-3}
|\partial_{t}^{n}D_{x}^{m}p(t,x)| \leq N |x|^{-d-m}t^{-n} \left(R  +R \ln R \cdot 1_{d=2,m=0}+ R^{1/2} \cdot 1_{d=1,m=0}\right).
\end{equation}
\end{lem}

\vspace{4mm}

 By \eqref{p-1},  $p(t,x)$ is absolutely continuous on $(0,T)$   and $\lim_{t\to 0}p(t,x)=0$ if $x\neq 0$. Thus we can define
 $$
 q(t,x):=\begin{cases}I^{\alpha-1}p(t,x), & \alpha\in(1,2) \\ D_{t}^{1-\alpha}p(t,x), &  \alpha\in(0,1). \end{cases}
$$
Since $p(0,x)=0$ for $x\neq 0$, $D_{t}^{1-\alpha}p(t,x)=\partial^{1-\alpha}_t p(t,x)$.

\begin{lem}
				\label{prop:kernel estimate}
(i)  Let $d\geq 1$, $\alpha\in (0,2)$, and $m,n=0,1,2,\cdots$. Denote $R=t^{-\alpha}|x|^2$.
Then there exist constants $N$ and $\sigma$ depending only on $m$, $n$, $d$, and $\alpha$ so that
 if $R\geq 1$
\begin{equation}
			\label{q-1}
|\partial_{t}^{n}D_{x}^{m}q(t,x)|\leq N t^{\frac{-\alpha(d+m)}{2}-n+\alpha-1}\exp\{-\sigma t^{-\frac{\alpha}{2-\alpha}}|x|^{\frac{2}{2-\alpha}}\},
\end{equation}
and if $R\leq 1$
\begin{align}
                    \notag
|\partial_{t}^{n} D_{x}^{m}q(t,x)| &\leq N|x|^{-d-m}t^{-n+\alpha-1}(R^2+R^2 \ln R \cdot 1_{d=2})\\
                     \label{q-3}
& \quad + N|x|^{-d}t^{-n+\alpha-1}\left( R^{1/2} \cdot 1_{d=1}+ R \cdot 1_{d=2}+R^2 \ln R \cdot 1_{d=4} \right) 1_{m = 0}.
\end{align}

(ii) For any $t\neq 0$ and $x\neq 0$,
\begin{equation}
                   \label{p equal}
\partial_{t}^{\alpha}p=\Delta p,
\quad \quad
              \frac{\partial p}{\partial t}=\Delta q.
\end{equation}
\end{lem}

 \vspace{3mm}

One can find  similar statements of Lemmas \ref{p exists} and  \ref{prop:kernel estimate} in
\cite{EIK, eidelman2004cauchy, kochubei2014asymptotic, pskhu2009fundamental}. For the sake of completeness, we give an independent and rigorous proof in Section \ref{p q kernel sec}.

\begin{corollary}
                        \label{integ cor-1}
Let $0<\varepsilon<T$. Then
$$
\int_{\mathbf{R}^{d}}\sup_{t\in [\varepsilon,T]}|D_{x}^{m}q(t,x)|dx<\infty,\quad \quad m=0,1,2.
$$
and
$$
\int_{\mathbf{R}^{d}}\sup_{t\in [\varepsilon,T]} |D^m_tp(t,x)| dx<\infty,  \quad \quad  m=0,1.
$$
\end{corollary}

\begin{proof}
By \eqref{q-1} and \eqref{q-3}, it follows
that for large $|x|$
 $$
 \sup_{t\in [\varepsilon,T]}|D^m_x q(t,x)|\leq
Ne^{- c |x|^{\frac{2}{2-\alpha}}}
 $$
and for small $|x|$
 $$
 \sup_{t\in [\varepsilon,T]}|D^m_x q(t,x)|\leq
N|x|^{-d-m}\left(|x|^3 +|x|\cdot 1_{m=0}\right),
 $$
where $N$ and $c$ depend only on $d$, $\alpha$, $T$, and $\varepsilon$.
This certainly proves the  assertion related to $q$, and $p$ is handled similarly.
The corollary is proved.
\end{proof}

\begin{lem}
            \label{lem 17.1}
Let $\lambda >0$, $\alpha\in (0,2)$, and $\phi$ be a continuous function on $[0,\infty)$ so that
the Laplace transforms of $\phi$ and $\partial^{\alpha}_t\phi$ exist,
\begin{equation}
             \label{9.16.4}
\partial^{\alpha}_t \phi +\lambda \phi=f(t), \quad t>0,
\end{equation}
and $\phi(0)=0$ $($additionally $\phi'(0)=0$ if $\alpha \in (1,2))$.
Moreover we assume for each $s>0$,
$$
e^{-st}\int_0^t\partial^{\alpha}_t \phi(r)dr \to 0 \quad \text{as} \quad t \to \infty.
$$
Then
\begin{equation}
           \label{eqn 9.19.1}
\phi(t)=\int^t_0 H_{\alpha,\lambda}(t-s)f(s)ds,
\end{equation}
where $H_{\alpha,\lambda}(t)=D^{1-\alpha}_t E_{\alpha}(-\lambda t^{\alpha})$ if $\alpha\in (0,1)$ and $H_{\alpha,\lambda}(t)=I^{\alpha-1}E_{\alpha}(-\lambda t^{\alpha})$ otherwise.
  \end{lem}

\begin{proof}
First, recall that $\varphi(t):=E_{\alpha}(-t^{\alpha})$ satisfies $\varphi(0)=1$ and $\partial^{\alpha}_t \varphi=-\varphi$.
Let $\cL$ denote the Laplace transform. Then from $\mathcal{L}[\partial^{\alpha}_t\varphi]=-\mathcal{L}[\varphi]$ we  get
$$
\mathcal{L}[\varphi](s):=\int^{\infty}_0 e^{-st}\varphi(t)dt=\frac{s^{\alpha-1}}{s^{\alpha}+1}.
$$
Here, we used the following facts: for $\beta\in (0,1)$,
\begin{equation}
              \label{eqn 9.22.1}
 \mathcal{L}[h']=s\mathcal{L}[h](s), \quad  \mathcal{L}[h*g](s)=\mathcal{L}[h](s) \cdot \mathcal{L}[g](s),  \quad \mathcal{L}[k_{1-\beta}](s)=s^{\beta-1}.
\end{equation}
It  follows that
\begin{equation}
        \label{16.5}
\mathcal{L}[E_{\alpha}(-\lambda t^{\alpha})](s)=\frac{s^{\alpha-1}}{s^{\alpha}+\lambda}, \quad
\mathcal{L}[H_{\alpha,\lambda}](s)=\frac{1}{s^{\alpha}+\lambda}.
\end{equation}

On the other hand, taking the Laplace transform to (\ref{9.16.4}) and using \eqref{eqn 9.22.1} we   get
$$
\mathcal{L}[\phi](s)=\frac{1}{s^{\alpha}+\lambda} \cdot \mathcal{L}[f](s), \quad s>0.
$$
This and (\ref{16.5}) certainly prove the lemma, because to prove equality \eqref{eqn 9.19.1} it is enough to show that two functions under consideration have the same Laplace transform.
\end{proof}
\begin{lem}
                  \label{lem 9.21.1}
(i) Let $u\in C_c^\infty(\fR^{d+1}_+)$ and denote $f:=\partial^{\alpha}_t u-\Delta u$. Then
\begin{equation}
                   \label{eqn 09.16.1}
                   u(t,x)=\int^t_0 \int_{\fR^d} q(t-s,x-y)f(s,y)dyds.
                   \end{equation}
 (ii) Let $f\in C_c^\infty(\fR^{d+1}_+)$ and define $u$ as in \eqref{eqn 09.16.1}. Then $u$ satisfies $\partial^{\alpha}_tu=\Delta u+f$.
  \end{lem}
  \begin{proof}
 (i) Since  $q$ is integrable on $(0,T)\times \fR^d$ (see Lemma \ref{prop:kernel estimate}), it is enough to prove
  \begin{equation}
             \label{eqn 16.2}
  \hat{u}(t,\xi)=\int^t_0 \hat{q}(t-s,\xi)\hat{f}(s,\xi)ds,
  \end{equation}
where $\hat{f}$ denotes the Fourier transform of $f$ with respect to $x$, i.e.
$$
\hat f(s,\xi):= \cF(f(s,\cdot))(\xi)= \frac{1}{(2\pi)^{d/2}} \int_{\fR^{d}} e^{-i x \cdot \xi} f(s,x) dx.
$$
First note that from the definition of $f$ we get
$$
\hat{f}(t,\xi)=\partial^{\alpha}_t \hat{u} (t,\xi)+|\xi|^2 \hat{u}(t,\xi), \quad \forall t>0.
$$
Therefore by Lemma \ref{lem 17.1}
$$
\hat{u}(t,\xi)=\int^t_0 H_{\alpha,|\xi|^2}(t-s)\hat{f}(s,\xi)ds.
$$
Hence it is enough to prove
\begin{equation}
         \label{10.13.1}
\hat{q}(t,\xi)=H_{\alpha,|\xi|^2}(t).
\end{equation}
Denote $c_d=(2\pi)^{-d/2}$. If $\alpha \in (0,1)$ then by the definition
$$
\hat{q}(t,x)=c_d \int_{\fR^d}e^{-i x\cdot \xi} q(t,x) dx=c_d\int_{\fR^d}e^{-i x\cdot \xi} \left[\frac{d}{dt}\int^t_0k_{\alpha}(t-s)p(s,x)ds \right] dx.
$$
Since $q(t,\cdot)$ is integrable in $\fR^d$ uniformly in a neighborhood of $t>0$ (see Corollary \ref{integ cor-1}), one can take the derivative $\frac{d}{dt}$ out of the integral.
After this using Fubini's theorem we get
$$
\hat{q}=\frac{d}{dt}\int^t_0k_{\alpha}(t-s)\left[c_d\int_{\fR^d}e^{-i x\cdot \xi} p(s,x)dx\right]ds=D^{1-\alpha}_tE_{\alpha}(-t^{\alpha}|\xi|^2).
$$
Hence \eqref{10.13.1} and (i) are proved. The case $\alpha\in [1,2)$ is easier and we skip the proof.

\vspace{3mm}

(ii) Taking the Fourier transform to \eqref{eqn 09.16.1} we get
$$
\hat{u}(t,\xi)=\int^t_0 H_{\alpha,|\xi|^2}(t-s) \hat{f}(s,\xi)ds.
$$
Note that if one defines $\phi$ as in \eqref{eqn 9.19.1} then it satisfies \eqref{9.16.4}. Consequently, $\hat{u}$ satisfies
$$
\partial^{\alpha}_t\hat{u}(t,\xi)+|\xi|^2\hat{u}(t,\xi)=\hat{f}(t,\xi),
$$
and this certainly proves the equality  $\partial^{\alpha}_tu=\Delta u+f$,
 because $\hat{f}(t,\cdot)=0$ if $t$ is small enough and thus for each $t>0$,
$$
\cF\left(\partial^{\alpha}_t u(t, \cdot)\right)(x) =\int^t_0 H_{\alpha,|\xi|^2}(t-s) \partial^\alpha_s\hat{f}(s,\xi)ds
=\partial^{\alpha}_t\hat{u}(t, \xi).
$$
The lemma is proved.
\end{proof}

 Now we define the operator $\mathcal{G}$ by
$$
\mathcal{G}f(t,x)=\int_{-\infty}^{t}\int_{\mathbf{R}^{d}}q(t-s,x-y)\Delta f(s,y)dyds.
$$
Since $q$ is integrable on $(0,T)\times \fR^d$, $\mathcal{G} f$ is well-defined if $f\in C_c^\infty(\mathbf{R}^{d+1})$. Also recall that $D_xq(t,\cdot)$ and $D^2_x q(t,\cdot)$ are integrable in $\fR^d$ for each $t>0$, and therefore it follows that
\begin{equation}
         \label{eqn 9.19.10}
\mathcal{G}f(t,x)=\lim_{\varepsilon \to 0}\int_{-\infty}^{t-\varepsilon}\left[\int_{\mathbf{R}^{d}}\Delta q(t-s,x-y) f(s,y)dy\right]ds.
\end{equation}

\begin{lem}
				\label{lem:L2}
Let $f\in C_c^\infty(\mathbf{R}^{d+1})$. Then
\begin{equation}
                  \label{9.19.3}
\left\Vert \mathcal{G}f\right\Vert _{L_{2}(\mathbf{R}^{d+1})}\leq N\|f\|_{L_{2}(\mathbf{R}^{d+1})}
\end{equation}
where $N=N(\alpha,d)$.
\end{lem}

\begin{proof}

Denote  $q_M=q 1_{0<t<M}$. One can easily check that $q_M$ is integrable in $\fR^{d+1}$. Denote $\cG_M f=q_M*\Delta f$. Then  by Parseval's identity
$$
\|\cG_M f\|^2_{L_2}=\int_{\mathbf{R}^{d+1}}|\mathcal{F}_{d+1} (q_M*\Delta f) |^2 d\tau d\xi.
$$
By the properties of the Fourier transform,
\begin{align*}
\mathcal{F}_{d+1} (q_M*\Delta f)(\tau,\xi) = -N(d)|\xi|^2\cF_{d+1} (q_M) (\tau,\xi) \cF_{d+1} (f) (\tau,\xi), \quad \quad \forall (\tau,\xi) \in \fR^{d+1}.
\end{align*}

Set
$$
I_M(\tau,\xi):=-|\xi|^2\cF_{d+1} (q_M) (\tau,\xi)
=\int_{0<t< M} e^{-i\tau t} \cF_d(\frac{\partial p}{\partial t})(t,\xi)dt.
$$
Now we claim  that $I_M(\tau,\xi)$ is bounded uniformly for $M>0$, i.e.
$$
\sup_{M>0,\tau,\xi} |I_M(\tau,\xi)| < \infty.
$$
Then the claim implies
\begin{equation}
              \label{eqn 11.14.1}
\|\cG_M f\|^2_{L_2}\leq N\|f\|^2_{L_2},
\end{equation}
where $N$ is independent of $M$.

By the integration by parts and the change of variables,
\begin{align}
                    \label{IM 1}
I_M(\tau,\xi)=i \text{sgn}(\tau) \int^{\tau M}_0 e^{- \text{sgn} (\tau) it}E_{\alpha}(-(\frac{t}{|\tau|})^{\alpha}||\xi|^2) dt+e^{-i\tau M} E_\alpha(-M^\alpha |\xi|^2)-1.
\end{align}
Denote for $\frac{\alpha\pi}{2}<\eta<\min\{\pi,\alpha\pi\}$
$$
\Delta^*_{\eta}:=\{z\in \mathbb{C}: |\pi-\text{Arg} \,z|<\pi-\eta\}.
$$
Then by \cite[(1.3.12)]{Dj} or \cite[Theorem 1.6]{Po},
$$
|E_{\alpha}(z)|=\frac{N}{1+|z|} , \quad z\in \Delta^*_{\eta}.
$$
Hence the function $E_{\alpha}(z)$ is bounded in $\Delta^*_{\eta}$.

Let $\tau>0$. We take a $\theta_0\in (0,\pi/3)$ sufficiently small so that $-e^{-i \theta},-e^{-i\alpha \theta}\in \Delta^*_{\eta}$ for any $\theta\in [0,\theta_0]$. Denote
$$
C_{1,\tau M}=\{t: t\in [0,\tau M]\}, \quad C_{2,\tau M}=\{te^{-i\theta_0}: t\in [0,\tau M]\},
$$
$$
 C_{3,\tau M}=\{\tau M e^{-i\theta}: \theta\in [0,\theta_0]\},
$$
and define a contour $C_{\tau M}=C_{1,\tau M} \cup C_{2,\tau M} \cup C_{3,\tau M}$.
Then the contour integral of the function $e^{-iz}E_{\alpha}(-(\frac{z}{\tau})^{\alpha}|\xi|^2)$ on $C_{\tau M}$ is zero for any $M>0$.
Since $E_{\alpha}(z)$ is bounded in $\Delta^*_{\eta}$,
\begin{align}
                    \label{IM 2}
\left|\int_{C_{3,\tau M}} e^{-iz}E_{\alpha}(-(\frac{z}{\tau})^{\alpha}|\xi|^2) dz\right|
&\leq \int^{\theta_0}_0 (\tau M) e^{-\tau M\sin \theta} d\theta \leq N(\theta_0).
\end{align}
Also,
\begin{equation}
               \label{eqn 10.16.1}
\left|\int_{C_{2,\tau M}} e^{-iz}E_{\alpha}(-(\frac{z}{\tau})^{\alpha}|\xi|^2) dz\right|
\leq N\int^{\tau M}_0 e^{-t\sin \theta_0}dt\leq N(\theta_0).
\end{equation}
Note that $\theta_0$ depends only on $\alpha$.
Hence by \eqref{IM 1}, \eqref{IM 2}, and \eqref{eqn 10.16.1}, it follows that if $\tau>0$ then $|I_{M}(\tau,\xi)|$ is bounded uniformly for $M$.

   If $\tau<0$ then we choose a contour $C'_{\tau M}=C_{1,\tau M}\cup C'_{2,\tau M}\cup C'_{3,\tau M}$ where
$$
C'_{2,\tau M}=\{te^{i\theta_0}: t\in [0,\tau M]\}, \quad C'_{3,\tau M}=\{\tau Me^{i\theta}: \theta\in [0,\theta_0]\}.
$$
Then the same arguments above go through. Thus our  claim is proved.

To finish the proof, observe that for each $(t,x) \in \bR^{d+1}$
$$
q_M*\Delta f(t,x) =\cG_M f(t,x) \to \cG f (t,x) \quad \text{as} \quad M \to \infty
$$
since $\Delta f \in C_c^{\infty}(\fR^{d+1})$.
Therefore
\eqref{eqn 11.14.1} and Fatou's lemma easily yields \eqref{9.19.3}.
\end{proof}

\mysection{BMO and $L_q(L_q)$-estimate}
                         \label{section model}
In this section we obtain BMO and $L_q(L_p)$-estimates for a solution of  the model equation
$$
\partial_{t}^{\alpha}u=\Delta u+f,  \quad  (t,x)\in(0,\infty)\times\mathbf{R}^{d}.
$$

Recall that $p$ is integrable with respect to $x$,  $\cF( p(t,\cdot))=E_{\alpha}(-t^{\alpha}|\xi|^2)$, and $q$ is defined as
$$
 q(t,x)=\begin{cases}I^{\alpha-1}p(t,x), & \alpha\in(1,2) \\ D_{t}^{1-\alpha}p(t,x), &  \alpha\in(0,1). \end{cases}
$$
Also recall that for $f\in C_c^\infty(\fR^{d+1}_+)$, it holds that
\begin{equation}
            \label{9.19.11}
\mathcal{G}f=\int_{-\infty}^{t}\left[\int_{\mathbf{R}^{d}}\Delta q(t-s,x-y) f(s,y)dy\right]ds,
\end{equation}
and by \eqref{9.19.3} the operator $\cG$ is continuously extended  onto $L_2(\fR^{d+1})$.  We   denote this extension  by the same notation $\cG$.

For a locally integrable function $h$ on $\mathbf{R}^{d+1}$, we
define the BMO semi-norm of $h$ on $\mathbf{R}^{d+1}$ as
$$
\|h\|_{BMO(\mathbf{R}^{d+1})}=\sup_{Q\in\mathbb{Q}}\frac{1}{|Q|}\int_{Q}|h(t,x)-h_{Q}|dtdx
$$
where $h_{Q}=\frac{1}{|Q|}\int_{Q}h(t,x)dtdx$ and
$$
\mathbb{Q}:=\{Q_{\delta}(t_{0},x_{0})=(t_{0}-\delta^{2/\alpha},t_{0}+\delta^{2/\alpha})\times B_{\delta}(x_{0})\ :\ \delta>0,~(t_{0},x_{0})\in\mathbf{R}^{d+1}\}.
$$
Denote $Q_{\delta}:=Q_{\delta}(0,0)$.

\begin{lem}\label{lem:pu0}
Let $f\in L_{2}(\mathbf{R}^{d+1})$ and vanish on $\mathbf{R}^{d+1}\setminus Q_{3\delta}$.
Then
$$
\aint_{Q_{\delta}}|\mathcal{G}f(t,x)|dtdx\leq N\|f\|_{L_{\infty}(\mathbf{R}^{d+1})},
$$
where $N=N(d,\alpha)$.
\end{lem}
\begin{proof}
By H\"{o}lder's inequality and Lemma \ref{lem:L2},
\begin{eqnarray*}
&&\int_{Q_\delta}|\mathcal{G}f(t,x)|dtdx  \leq\left(\int_{Q_\delta}|\mathcal{G}f(t,x)|^{2}dtdx\right)^{1/2}|Q_\delta|^{1/2}\\
 && \leq\|\mathcal{G}f\|_{L_{2}(\mathbf{R}^{d+1})}|Q_\delta|^{1/2}
 \leq N\|f\|_{L_{2}(\mathbf{R}^{d+1})}|Q_\delta|^{1/2}\\
 && =N\left(\int_{Q_{3\delta}}|f(t,x)|^{2}dtdx\right)^{1/2}|Q_\delta|^{1/2}
 \leq N\|f\|_{L_{\infty}(\mathbf{R}^{d+1})}|Q_\delta|.
\end{eqnarray*}
The lemma is proved.
\end{proof}

Denote
$$
K(t,x)=1_{t>0}\Delta q(t,x).
$$
Due to Lemma \ref{prop:kernel estimate}(ii), we have $K(t,x)=1_{t>0}\partial_{t}p(t,x)$.
Furthermore the following scaling properties hold (see \eqref{K equal} and \eqref{appendix:diff} for detail):
\begin{align}
			\label{scaling property1}
K(t,x)=1_{t>0}t^{-1-\frac{\alpha d}{2}}K(1,t^{-\alpha/2}x)			
\end{align}
\begin{align}
			\label{scaling property2}
\partial_{t}K(t,x)=1_{t>0}t^{-2-\frac{\alpha d}{2}}(\partial_{t}K)(1,t^{-\alpha/2}x),		
\end{align}
and
\begin{align}
			\label{scaling property3}
\frac{\partial}{\partial x^i} K(t,x)=1_{t>0}t^{-1-\frac{\alpha (d+1)}{2}} \frac{\partial}{\partial x^i} K(1,t^{-\alpha/2}x).			
\end{align}

\begin{lem}
				\label{lem:BMO assumption}
There exists a constant $N=N(\alpha,d)$ such that

(i) for any $t>a$ and $\eta>0$,
\begin{align}\label{K:ass1}
\int_{a}^{t}\int_{|y|\geq\eta}|K(t-s,y)|dyds \leq N (t-a)\eta^{-2/\alpha};
\end{align}

(ii) for any $t>\tau>a$,
\begin{align}\label{K:ass2}
\int_{-\infty}^{a}\int_{\mathbf{R}^{d}}|K(t-s,y)-K(\tau-s,y)|dyds \leq N\frac{t-\tau}{\tau-a};
\end{align}

(iii) for any $t>a$ and $x\in\mathbf{R}^{d}$,
\begin{align}\label{K:ass3}
\int_{-\infty}^{a}\int_{\mathbf{R}^{d}}|K(t-s,x+y)-K(t-s,y)|dyds \leq N|x|(t-a)^{-\alpha/2}.
\end{align}
\end{lem}
\begin{proof}
First observe that
\begin{align*}
			 \int_{\mathbf{R}^{d+1}}\left(|y|^{2/\alpha}|K(1,y)|+|D_x K(1,y)|+|\partial_{t}K(1,y)|\right)dy<\infty,
\end{align*}
which is an easy consequence of Lemma \ref{prop:kernel estimate}.
 By \eqref{scaling property1}, it holds that
\begin{align*}
\int_{a}^{t}\int_{|y|\geq\eta}|K(t-s,y)|dyds &= \int_{a}^{t}(t-s)^{-1-\frac{\alpha d}{2}}\int_{|y|\geq \eta}K(1,(t-s)^{-\alpha/2}y)dyds \\
	&=\int_{a}^{t}(t-s)^{-1}\int_{|y|\geq \eta(t-s)^{-\alpha/2}}K(1,y)dyds \\
	&\leq \left(\int_{a}^{t}\eta^{-2/\alpha}ds\right)\left(\int_{\mathbf{R}^{d}}|y|^{2/\alpha}K(1,y)dy\right)\\
	&\leq N(t-a)\eta^{-2/\alpha}.
\end{align*}
Hence  (i) is proved.

 Next
 we prove (ii) and (iii)   on the basis of the scaling property.
By \eqref{scaling property2},  \eqref{scaling property3}, and the mean-value theorem, we have
\begin{align*}
&\int_{-\infty}^{a}\int_{\mathbf{R}^{d}}|K(t-s,y)-K(\tau-s,y)|dyds \\
	&\leq (t-\tau)\int_{-\infty}^{a}\int_{\mathbf{R}^{d}}\int_{0}^{1}|\partial_s K(\theta t+(1-\theta)\tau-s,y)|d\theta dyds \\
	&\leq (t-\tau)\int_{0}^{1}\left(\int_{-\infty}^{a}(\theta t+(1-\theta)\tau-s)^{-2}ds\right)\left(\int_{\fR^d}|\partial_s K(1,y)|dy\right)d\theta\\
	&\leq (t-\tau)\left(\int_{-\infty}^{a}(\tau-s)^{-2}ds\right)\leq N\frac{t-\tau}{\tau-a}.
\end{align*}
Also,
\begin{align*}
&\int_{-\infty}^{a}\int_{\mathbf{R}^{d}}|K(t-s,y+x)-K(t-s,y)|dyds \\
		&\leq |x|\int_{-\infty}^{a}\int_{\mathbf{R}^{d}}\int_{0}^{1}|D_{y}K(t-s,y+\theta x)|d\theta dyds \\
		&\leq |x|\int_{t-a}^{\infty}\int_{\mathbf{R}^{d}}|D_{y}K(s,y)|dyds \\
		&\leq |x|\int_{t-a}^{\infty}\int_{\mathbf{R}^{d}}|1_{t>0}t^{-1-\frac{\alpha }{2}} D_y K(1,y))|dyds \leq N|x|(t-a)^{-\alpha/2}.
\end{align*}
The lemma is proved.
\end{proof}

\begin{lem}
				\label{lem:pu1}
Let $f\in L_{2}(\mathbf{R}^{d+1})$ and $f=0$ on $Q_{2\delta}$. Then
\begin{align}
                \label{mean osc est}
\aint_{Q_{\delta}}\aint_{Q_{\delta}}|\mathcal{G}f(t,x)-\mathcal{G}f(s,y)|dsdydtdx\leq N(d,\alpha)\|f\|_{L_{\infty}(\mathbf{R}^{d+1})}.
\end{align}
\end{lem}
\begin{proof}

First  assume $f \in C_c^\infty (\fR^{d+1})$. We claim that
\begin{align}
                    \label{mean osc in zero 0}
\aint_{Q_{\delta}} |  \mathcal{G} f (t,x)- \mathcal{G} f (-\delta^{2/\alpha},0)|dt dx  \leq   N \| f\|_{L_{\infty}(\mathbf{R}^{d+1})}.
\end{align}
Let $(t,x) \in Q_\delta$. Then
\begin{align*}
&|\mathcal{G} f (t,x)- \mathcal{G} f (-\delta^{2/\alpha},0)|\\
&\leq | \mathcal{G} f (t,x)-  \mathcal{G} f(t,0)| + | \mathcal{G} f(t,0)- \mathcal{G} f (-\delta^{2/\alpha},0)|=: \cI_1 + \cI_2.
\end{align*}

We consider $\cI_1$ first.
\begin{align*}
\cI_1
&= \left| \int_{-\infty}^{t} \int_{\fR^d}  \big(K(t-s,x-y)-K(t-s,-y)\big) f(s,y)dy ds\right| \\
&= \left|\int^{t}_{-(2\delta)^{2/\alpha}}\int_{\fR^d} \cdots dyds +\int_{-\infty}^{-(2\delta)^{2/\alpha}} \int_{\fR^d}\cdots dyds \right|\\
&\leq  \int_{-(2\delta)^{2/\alpha}}^{t} \int_{\fR^d}  |K(t-s,y)| |f(s, x-y)|dy ds \\
&\quad+ \int_{-(2\delta)^{2/\alpha}}^{t} \int_{\fR^d}  |K(t-s,y)| |f(s ,-y)|dy ds  \\
&\quad+\int_{-\infty}^{-(2\delta)^{2/\alpha}} \int_{\fR^d}  \left|K(t-s,x-y)-K(t-s,-y)\right| |f(s, y)|dy ds \\
&\quad=:\cI_{11}+\cI_{12}+\cI_{13}.
\end{align*}
Note that if $-(2\delta)^{2/\alpha} < s  \leq t < \delta^{2/\alpha}$ and $|y| \leq \delta$, then
\begin{align}           \label{zero con}
f(s,  x-y) =0,\quad f(s,-y)=0,
\end{align}
 because $|x-y| \leq 2\delta$, $|-y| \leq \delta$, and $f=0$ on $Q_{2\delta}$.
Hence by \eqref{K:ass1}, $\cI_{11}+\cI_{12}$ is less than or equal to
\begin{align*}
&N \|f\|_{L_{\infty}(\mathbf{R}^{d+1})} \int_{-(2\delta)^{2/\alpha}}^{t} \int_{|y| \geq \delta} | K(t-s,y)|dyds\\
&\leq N \|f\|_{L_{\infty}(\mathbf{R}^{d+1})} (t+(2\delta)^{2/\alpha})\delta^{-2/\alpha} \leq N \|f\|_{L_{\infty}(\mathbf{R}^{d+1})}.
\end{align*}
Also, by \eqref{K:ass3} we have
\begin{align*}
\cI_{13}&\leq N \|f\|_{L_{\infty}(\mathbf{R}^{d+1})}  \int_{-\infty}^{-(2\delta)^{2/\alpha}} \int_{\fR^d} \left|K(t-s, x-y)-K(t-s,-y)\right|dyds\\
&\leq N \|f\|_{L_{\infty}(\mathbf{R}^{d+1})} |x| (t+(2\delta)^{2/\alpha})^{-\alpha/2} \leq N \|f\|_{L_{\infty}(\mathbf{R}^{d+1})}.
\end{align*}

Next, we consider $\cI_2$. Note that
\begin{align*}
&\cI_2=\left|  \mathcal{G} f (t,0) -  \mathcal{G} f(-\delta^{2/\alpha},0) \right| \\
&= \left| \int_{-\infty}^{t} \int_{\fR^d}  K(t-s, y)  f(s,-y)dy ds - \int_{-\infty}^{-\delta^{2/\alpha}} \int_{\fR^d}  K(-\delta^{2/\alpha}-s,-y)  f(s,-y)dyds \right|\\
&\leq \left| \int_{-\infty}^{t} \int_{\fR^d} K(t-s, y)  f(s,-y)dsdy - \int_{-\infty}^{-\delta^{2/\alpha}} \int_{\fR^d} K(t-s,y)  f(s,-y)dyds \right|\\
&+\left| \int_{-\infty}^{-\delta^{2/\alpha}} \int_{\fR^d} \left[K(t-s,y)-K(-\delta^{2/\alpha}-s, y) \right] f(s,-y)dy ds \right|=: \cI_{21} + \cI_{22}.
\end{align*}
Recall \eqref{zero con}. Thus by \eqref{K:ass1}, we have
\begin{align*}
\cI_{21}
&\leq  \int_{-\delta^{2/\alpha}}^{t} \int_{\fR^d} |K(t-s, y)f(s,-y)|dyds \\
&\leq  \|f\|_{L_{\infty}(\mathbf{R}^{d+1})} \int_{-\delta^{2/\alpha}}^{t} \int_{|y| \geq \delta} | K(t-s, y)|dyds \\
&\leq  N \|f\|_{L_{\infty}(\mathbf{R}^{d+1})}(t+\delta^{2/\alpha})\delta^{-2/\alpha}  \leq N\|f\|_{L_{\infty}(\mathbf{R}^{d+1})}.
\end{align*}
Also,
\begin{align*}
\cI_{22}
&\leq  \int_{-(2\delta)^{2/\alpha}}^{-\delta^{2/\alpha}} \int_{\fR^d} | K(t-s,y) - K(-\delta^{2/\alpha}-s,y)| |f(s,-y)|dy ds \\
&\quad+\|f\|_{L_{\infty}(\mathbf{R}^{d+1})} \int_{-\infty}^{-(2\delta)^{2/\alpha}} \int_{\fR^d} \left| K(t-s, y) - K(-\delta^{2/\alpha}-s, y) \right| dy ds \\
&=: \cI_{221}+ \cI_{222}.
\end{align*}
By \eqref{K:ass1} again, we have
\begin{align*}
\cI_{221}
&\leq  \|f\|_{L_{\infty}(\mathbf{R}^{d+1})} \int_{-(2\delta)^{2/\alpha}}^{t} \int_{|t| \geq \delta} | K(t-s, t)|dyds \\
&\quad+ \|f\|_{L_{\infty}(\mathbf{R}^{d+1})}\int_{-(2\delta)^{2/\alpha}}^{-\delta^{2/\alpha}} \int_{|y| \geq \delta} | K(-\delta^{2/\alpha}-s, y) | dy ds\\
&\leq N \|f\|_{L_{\infty}(\mathbf{R}^{d+1})}.
\end{align*}
On the other hand, by \eqref{K:ass2} we obtain
$$
\cI_{222} \leq \frac{t+\delta^{2/\alpha}}{-\delta^{2/\alpha}+(2\delta)^{2/\alpha}} \|f\|_{L_{\infty}(\mathbf{R}^{d+1})} \leq N \|f\|_{L_{\infty}(\mathbf{R}^{d+1})}.
$$
Hence \eqref{mean osc in zero 0} is proved and this obviously implies \eqref{mean osc est} for $f \in C_c^\infty(\fR^{d+1})$.

 Now we consider the general case, that is  $f \in L_2(\fR^{d+1})$.
 We choose a sequence of functions $f_n \in C_c^\infty(\fR^{d+1})$ such that $f_n=0$ on $Q_{2\delta}$,
$\mathcal{G} f_n \to \mathcal{G} f~(a.e.)$, and $ \|f_n\|_{L_{\infty}(\mathbf{R}^{d+1})} \leq \|f\|_{L_{\infty}(\mathbf{R}^{d+1})}$.
Then by Fatou's lemma,
\begin{align*}
&\aint_{Q_{\delta}}\aint_{Q_{\delta}} |  \mathcal{G} f (t,x)- \mathcal{G} f (s,y)|~dtds dxdy \\
&\leq \liminf_{n \to \infty} \aint_{Q_{\delta}} \aint_{Q_{\delta}} |  \mathcal{G} f_n (t,x)- \mathcal{G} f_n (s,y)|dtds dxdy \\
&\leq N\liminf_{n \to \infty} \|f_n\|_{L_{\infty}(\mathbf{R}^{d+1})} \leq N \|f\|_{L_{\infty}(\mathbf{R}^{d+1})}.
\end{align*}
The lemma is proved.
\end{proof}

\begin{thm}
                    \label{BMO theorem}
(i) For any $f \in L_{2}(\mathbf{R}^{d+1})\cap L_{\infty}(\mathbf{R}^{d+1})$,
\begin{equation}
                     \label{eq:main result}
\|\mathcal{G} f \|_{BMO(\mathbf{R}^{d+1})} \leq N(d,\alpha) \|f\|_{L_\infty(\mathbf{R}^{d+1})}.
\end{equation}

(ii) For any $p,q\in (1,\infty)$ and $f \in C_c^\infty(\fR^{d+1})$,
\begin{equation}
                 \label{Lp estimate}
\|\mathcal{G} f \|_{L_q( \mathbf{R},L_{p} )} \leq N(d,p,q,\alpha) \|f\|_{L_q( \mathbf{R},L_{p} )}.
\end{equation}
\end{thm}

\begin{proof}
(i) It suffices to prove that for each $Q=Q_{\delta}(t_{0},x_{0})$
the following holds :
\begin{equation}
             \label{eqn 11.03.1}
\aint_{Q}|\mathcal{G}f(t,x)-(\mathcal{G}f)_{Q}|dtdx \leq N\|f\|_{L_{\infty}(\mathbf{R}^{d+1})}.
\end{equation}
Due to the translation invariant property of the operator $\cG$, we may assume that $(t_{0},x_{0})=0$.
Thus
$$
Q=Q_{\delta}=(-\delta^{2/\alpha},\delta^{2/\alpha})\times B_{\delta}.
$$
Take $\zeta \in C_{c}^{\infty}(\mathbf{R}^{d+1})$ such that $\zeta=1$ on $Q_{2\delta}$ and $\zeta=0$ outside $Q_{3\delta}$. Then
\begin{align*}
&\aint_{Q}|\mathcal{G}f(t,x)-(\mathcal{G}f)_{Q}|dtdx \\
&\leq \aint_{Q}\left|\mathcal{G}(\zeta f)-(\mathcal{G}(\zeta f))_{Q}\right|dtdx +\aint_{Q}\left|\mathcal{G}((1-\zeta) f)-(\mathcal{G}((1-\zeta) f))_{Q}\right|dtdx\\
&\leq 2\aint_{Q}|\mathcal{G}(\zeta f)|dtdx+\aint_{Q}\aint_{Q}|\mathcal{G}((1-\zeta) f)(t,x)-\mathcal{G}((1-\zeta) f)(s,y)|dsdydtdx.
\end{align*}
Thus \eqref{eqn 11.03.1} comes from  Lemma \ref{lem:pu0} and Lemma \ref{lem:pu1}.

(ii) {\bf Step 1}.
We prove \eqref{Lp estimate} for the case $q=p$. First we assume $q=p\geq 2$. For a measurable function $h(t,x)$ on $\mathbf{R}^{d+1}$, we define the maximal function
$$
\mathcal{M} h(t,x)
= \sup_{Q\in \mathbb{Q}}\aint_{Q}|f(r,z)|drdz,
$$
and the sharp function
\begin{align*}
h^\sharp(t,x)
= \sup_{ Q \in \mathbb{Q}}\aint_{ Q} |f(r,z) - f_{ Q}|~drdz.
\end{align*}
 Then by the Fefferman-Stein theorem  and the Hardy-Littlewood maximal theorem
\begin{align}
                     \label{hardy littlewood}
\|h^\sharp\|_{L_p(\mathbf{R}^{d+1})} \sim \|\cM h\|_{L_p(\mathbf{R}^{d+1})} \sim \|h\|_{L_p(\mathbf{R}^{d+1})}.
\end{align}
Combining  Lemma \ref{lem:L2} with \eqref{hardy littlewood}, we get for  any $f \in L_{2}(\mathbf{R}^{d+1})\cap L_{\infty}(\mathbf{R}^{d+1})$,
$$
\|( \mathcal{G} f)^\sharp\|_{L_2(\mathbf{R}^{d+1})} \leq N \|f\|_{L_2(\mathbf{R}^{d+1})}.
$$
Also, by \eqref{eq:main result}
\begin{align*}
\|(\mathcal{G} f)^\sharp\|_{L_{\infty}(\mathbf{R}^{d+1})}\leq N \|f\|_{L_{\infty}(\mathbf{R}^{d+1})}.
\end{align*}
Note that the  map  $f \to (\mathcal{G} f)^\sharp$ is subadditive since $\mathcal{G}$ is a linear operator.
Hence by a version of  the Marcinkiewicz interpolation theorem (see  e.g. \cite[Lemma 3.4]{Kim2014BMOpseudo}),
 for any $p \in [2, \infty)$ there exists a constant $N$ such that
$$
\|( \mathcal{G} f)^\sharp\|_{L_p(\mathbf{R}^{d+1})} \leq N \|f\|_{L_p(\mathbf{R}^{d+1})},
$$
for all $f \in L_{2}(\mathbf{R}^{d+1})\cap L_{\infty}(\mathbf{R}^{d+1})$. Finally by the Fefferman-Stein theorem, we get
$$
\|\mathcal{G} f\|_{L_p(\mathbf{R}^{d+1})} \leq N(d,\alpha,p) \|f\|_{L_p(\mathbf{R}^{d+1})}.
$$
Therefore \eqref{Lp estimate} is proved for $q=p\in [2,\infty)$.

For $p\in (1,2)$, we use the duality argument. Let $f,g\in C_{c}^{\infty}(\mathbf{R}^{d+1})$ and $p'$ be the conjugate of $p$, i.e. $p'= \frac{p}{p-1} \in (2,\infty)$.
By  the integration by parts, the change of variable, and Fubini's theorem,
\begin{eqnarray}
&&\int_{\mathbf{R}^{d+1}}g(t,x)\mathcal{G}f(t,x)dxdt\nonumber\\
&&=\int_{\mathbf{R}^{d+1}}\Delta g(t,x)\left(\int_{\mathbf{R}^{d+1}}1_{t>s}q(t-s,x-y) f(s,y)dsdy\right)dxdt\nonumber\\
&&=\int_{\mathbf{R}^{d+1}}f(-s,-y)\left(\int_{\mathbf{R}^{d+1}}1_{s>t} q(s-t,y-x)\Delta g(-t,-x)dxdt\right)dsdy\nonumber\\
&&=\int_{\mathbf{R}^{d+1}}f(-s,-y)\mathcal{G}\tilde g(s,y)dsdy   \label{duality}
\end{eqnarray}
where $\tilde g(t,x)=g(-t,-x)$. Then by H\"{o}lder's inequality (also recall $2<p'<\infty$),
$$
\left|\int_{\mathbf{R}^{d+1}}g(t,x)\mathcal{G}f(t,x)dxdt\right| \leq N\|f\|_{L_{p}}\|\mathcal{G} \tilde g\|_{L_{p'}} \leq N\|f\|_{L_{p}}\|g\|_{L_{p'}}.
$$
Since $g\in C^{\infty}_c(\mathbf{R}^{d+1})$ is arbitrary, \eqref{eq:main result} is proved for $q=p \in (1,2)$ as well.

{\bf{Step 2}}. Now we prove \eqref{Lp estimate} for general $p,q\in(1,\infty)$.
For each $(t,s) \in \fR^2$, we define the operator $\mathcal{K}(t,s)$ as follows:
$$
\mathcal{K}(t,s)f(x):=\int_{\mathbf{R}^{d}}K(t-s,x-y)f(y)dy,\quad f\in C_{c}^{\infty}.
$$
Let $p\in (1,\infty)$ and $(t,s) \in \fR^2$. Then by the mean-value theorem, Lemma \ref{prop:kernel estimate}, and \eqref{scaling property1},
\begin{align*}
\|\mathcal{K}(t,s)f\|_{L_{p}}&=\left\|\int_{\mathbf{R}^{d}}K(t-s,x-y)f(y)dy\right\|_{L_{p}} \\
&\leq \|f\|_{L_{p}}\int_{\mathbf{R}^{d}}|K(t-s,y)|dy\leq N(d,\alpha)(t-s)^{-1}\|f\|_{L_{p}}.
\end{align*}
 Hence the operator $\mathcal{K}(t,s)$ is uniquely extendible to $L_{p}$ for $t\neq s$. Denote
$$
Q:=[t_{0},t_{0}+\delta),\quad Q^{*}:=[t_{0}-\delta,t_{0}+2\delta).
$$
Note that for $t\notin Q^{*}$ and $s,r\in Q$, we have
$$
|s-r|\leq \delta,\quad |t-(t_{0}+\delta)|\geq \delta,
$$
and recall $K(t-s,x-y)=0$ if $t\leq s$.
Thus by Lemma \ref{prop:kernel estimate} and \eqref{scaling property2}, it holds that
\begin{align*}
&\|\mathcal{K}(t,s)-\mathcal{K}(t,r)\|_{\fL(L_{p})} \\
&=\sup_{\|f\|_{L_{p}}=1}\left\|\int_{\mathbf{R}^{d}}\left\{K(t-s,x-y)-K(t-r,x-y)\right\}f(y)dy\right\|_{L_{p}} \\
&\leq \sup_{\|f\|_{L_{p}}=1}\|f\|_{L_{p}}\int_{\mathbf{R}^{d}}|K(t-s,x)-K(t-r,x)|dx \\
&=\int_{\mathbf{R}^{d}}|K(t-s,x)-K(t-r,x)|dx \leq \frac{N(d,\alpha)|s-r|}{(t-(t_{0}+\delta))^{2}},
\end{align*}
where $\fL(L_{p})$ denotes the bounded linear operator norm on $L_p$.
Therefore,
\begin{align*}
&\int_{\mathbf{R}\setminus Q^{*}}\|\mathcal{K}(t,s)-\mathcal{K}(t,r)\|_{L(L_{p})}dt  \leq \int_{\mathbf{R}\setminus Q^{*}} \frac{|s-r|}{(t-(t_{0}+\delta))^{2}} dt\\
&\leq N|s-r|\int_{|t-(t_0+\delta)| \geq \delta} \frac{1}{(t-(t_{0}+\delta))^{2}}dt \leq N\delta \int_{\delta}^{\infty}t^{-2}dt \leq N.
\end{align*}
Furthermore, by following the proof of Theorem 1.1 of \cite{krylov2001calderon}, one can easily check that for almost every $t$ outside of the  support of $f\in C_{c}^{\infty}(\mathbf{R},L_{p})$,
$$
\mathcal{G}f(t,x)=\int_{-\infty}^{\infty}\mathcal{K}(t,s)f(s,x)ds
$$
where $\mathcal{G}$ denotes the unique extension on $L_{p}(\mathbf{R}^{d+1})$ which is verified in Step 1.
Hence by the Banach space-valued version of the Calder\'{o}n-Zygmund theorem \cite[Theorem 4.1]{krylov2001calderon}, our assertion is proved for $1<q\leq p$.

For the remaining case that $1<p<q<\infty$, we use the duality argument. Define $p'=p/(p-1)$ and $q'=q/(q-1)$.
Since $1<q'<p'$, by \eqref{duality} and H\"older's inequality,
\begin{align*}
\int_{\mathbf{R}^{d+1}}g(t,x)\mathcal{G}f(t,x)dxdt &= \int_{\mathbf{R}}\left(\int_{\mathbf{R}^{d}}f(-s,-y)\mathcal{G}\tilde{g}(s,y)dy\right)ds\\
&\leq \int_{\mathbf{R}}\|f(-s,\cdot)\|_{L_{p}}\|\mathcal{G}\tilde{g}(s,\cdot)\|_{L_{p'}}ds \\
&\leq N(d,\alpha,p) \|f\|_{L_{q}(\mathbf{R},L_{p})}\|g\|_{L_{q'}(\mathbf{R},L_{p'})}
\end{align*}
for any $f,g\in C_{c}^{\infty}(\mathbf{R}^{d+1})$. Since $g$ is arbitrary, we have
$$
\|\mathcal{G}f\|_{L_{q}(\mathbf{R},L_{p})} \leq N(d,\alpha,p)\|f\|_{L_{q}(\mathbf{R},L_{p})}.
$$
The theorem is proved.
\end{proof}

\mysection{Proof of Theorem \ref{main theorem}}
                    \label{pf main thm}

First we consider the solvability of the model equation.
\begin{lem}
                              \label{lem 9.21.6}
Theorem \ref{main theorem} holds for the equation $\partial^{\alpha}_tu=\Delta u+f_0$ with $N_0=N_0(p,q,\alpha,T)$.
\end{lem}

\begin{proof}
First we prove the uniqueness. Suppose that $u \in \bH^{\alpha,2}_{q,p,0}(T)$ and $\partial^{\alpha}_tu=\Delta u$.
Due to Lemma \ref{lem 9.21.11} (iii), there exists a $u_n \in C_c^\infty(\fR^{d+1})$ such that
$u_n \to u$ in $\bH^{\alpha,2}_{q,p,0}(T)$.
Due to Lemma \ref{lem 9.21.1} (i),
$$
u_n(t,x)=\int^t_0 \int_{\fR^d} q(t-s,x-y)f_n(s,y)dyds,
$$
where $f_n:=\partial^{\alpha}_t u_n-\Delta u_n$. Since
$f_n \to 0 $ in $\bL^{q,}_{p}(T)$,
we obtain the uniqueness from Lemma \ref{lem 9.21.11} (iv) and Theorem \ref{BMO theorem}.

For the existence and \eqref{main estimate}, first assume $f\in C_c^\infty(\fR^{d+1}_+)$ and define $u=\mathcal{G} f$.
Then all the claims follows from Lemma \ref{lem 9.21.1}(ii),  Theorem \ref{BMO theorem}, and \eqref{eqn 9.21}.
For general $f$ one can consider an approximation $f_n \to f$, and  above arguments show that $\mathcal{G} f_n$ is a Cauchy sequence in $\bH^{\alpha,2}_{q,p,0}(T)$ and the limit  becomes a solution of $\partial^{\alpha}_tu=\Delta u+f$.
\end{proof}


The following two results will be used later when we extend results proved for small $T$ to the case when $T$ is arbitrary.
\begin{lem}
                      \label{cor 9.21.10}
Let $u\in \bH^{\alpha,2}_{q,p,0}(\tilde T)$ and $\tilde{T}\leq T$.
Then there exists $\tilde{u}\in \bH^{\alpha,2}_{q,p,0}(T)$ such that $\tilde{u}(t)=u(t)$ for all $t\leq \tilde{T}$, and
\begin{align}
                \label{ext est}
\|\tilde{u}\|_{\bH^{\alpha,2}_{q,p}(T)}\leq N_0 \|u\|_{\bH^{\alpha,2}_{q,p}(\tilde{T})},
\end{align}
where $N_0$ is from Lemma \ref{lem 9.21.6} and is independent of $\tilde T$.
\end{lem}
\begin{proof}
Denote $f=\partial^{\alpha}_t u$, and let $\tilde{u}\in \bH^{\alpha,2}_{q,p,0}(T)$ be the solution of
$$
\tilde{u}^{\alpha}_t=\Delta \tilde{u}+(f-\Delta u)1_{t\leq \tilde T}, \quad t\leq T.
$$
Then by Lemma \ref{lem 9.21.6},
$$
\|\tilde{u}\|_{\bH^{\alpha,2}_{q,p}(T)}\leq N_0\|(f-\Delta u)1_{t\leq \tilde T}\|_{\bL_{q,p}(T)}\leq N_0\|u\|_{\bH^{\alpha,2}_{q,p}(\tilde T)}.
$$
Next observe that for $t\leq \tilde T$,
$$
\partial^{\alpha}_t(\tilde{u}-u)=\Delta \tilde{u}+f-\Delta u -f=\Delta (\tilde{u}-u).
$$
It follows from Lemma \ref{lem 9.21.6} that $\tilde{u}=u$ for $t\leq \tilde T$. The lemma is proved.
\end{proof}

\begin{lem}
                    \label{restric zero}
Let $0<\tilde T < T$ and $u,\tilde u \in \bH^{\alpha,2}_{q,p,0}(T)$.
Assume that
\begin{align}
                    \label{equal con}
u(t)=\tilde u(t) \quad \quad t \leq \tilde T \quad (a.e.).
\end{align}
Then
$\bar u (t):= u(\tilde T +t )-\tilde u(\tilde T +t ) \in \bH^{\alpha,2}_{q,p,0}(T-\tilde T)$.
\end{lem}
\begin{proof}
We take $u^{\varepsilon_1 ,\varepsilon_2, \varepsilon_3}$ from the proof of Lemma \ref{lem 9.21.11} (iii) which is defined as
$$
u^{\varepsilon_1 ,\varepsilon_2, \varepsilon_3}(t,x)
=\eta(t)\eta_{3}(\varepsilon_3 x)\int_{\fR^d}\int_0^\infty u(s,y) \eta_{1,\varepsilon_1}(t-s)\eta_{2,\varepsilon_2}(x-y)dsdy.
$$
As shown before, for any $\varepsilon >0$ it holds that
\begin{align*}
&\|u-u^{\varepsilon_1,\varepsilon_2,\varepsilon_3}\|_{\bH^{\alpha,2}_{q,p}(T)} \leq \varepsilon
\end{align*}
if $\varepsilon_1$, $\varepsilon_2$, and $\varepsilon_3$ are small enough.
Hence we can take a sequence $(a_n,b_n,c_n)$ so that
\begin{align*}
&\|u-u^{a_n,b_n,c_n}\|_{\bH^{\alpha,2}_{q,p}(T)} \to 0 \quad \text{as} \quad n \to \infty
\end{align*}
and
\begin{align*}
&\|\tilde u-\tilde u^{a_n,b_n,c_n}\|_{\bH^{\alpha,2}_{q,p}(T)} \to 0 \quad \text{as} \quad n \to \infty.
\end{align*}
Observe that
$$
u^{a_n,b_n,c_n}(t,x)=\tilde u^{a_n,b_n,c_n}(t,x) \quad \quad \forall t \leq (\tilde T+a_n) \wedge T
$$
due to \eqref{equal con} and the fact that $\eta_{1,\varepsilon_1} \in C_c^\infty( (\varepsilon_1, 2\varepsilon_1))$.
Thus
$$
\bar u_n(t,x) := u^{a_n,b_n,c_n}(\tilde T+t,x)-\tilde u^{a_n,b_n,c_n}(\tilde T+t,x)
$$ is a defining sequence of $\bar u$ such that
$$
\bar u_n(0,x)= 0  \quad \text{and} \quad \frac{\partial}{\partial t}\bar{u}_n(0,x)=0 \quad \quad \forall x \in \fR^d, \quad \forall n \in \bN.
$$
Therefore the lemma is proved.
\end{proof}

\vspace{2mm}
{\bf{Proof of Theorem \ref{main theorem}}}.
\vspace{2mm}

{\bf{Step 1}}. Denote $\bar{f}=b^iu_{x^i}+cu+f(u)$. Then
\begin{align*}
\|\bar{f}(u)-\bar{f}(v)\|_{L_p}&\leq N\|u-v\|_{H^1_p}+\|f(u)-f(v)\|_{L_p}\\
&\leq \varepsilon \|u-v\|_{H^2_p}+K\|u-v\|_{L_p}+ \|f(u)-f(v)\|_{L_p}.
\end{align*}
Thus considering $\bar{f}$ in place of $f$ we may assume $b^i=c=0$.

{\bf{Step 2}}. Let $f=f_0$ be independent of $u$.
Assume that $a^{ij}$ are independent of $(t,x)$.
In this case obviously we may assume $a^{ij}=\delta^{ij}$, and therefore the results follow from Lemma \ref{lem 9.21.6}.

{\bf{Step 3}}. Let $f=f_0$.
Suppose that Theorem \ref{main theorem} holds with some matrix $\bar{a}=(\bar{a}^{ij}(t,x))$ in place $(a^{ij}(t,x))$.
We prove that there exists $\varepsilon_0=\varepsilon_0(N_0)>0$ so that if
$$
\sup_{(t,x)}|a^{ij}(t,x)-\bar{a}^{ij}(t,x)|\leq \varepsilon_0,
$$
then Theorem \ref{main theorem} also holds for $(a^{ij}(t,x))$ with $2N_0$ in place of $N_0$. To prove this, due to the method of continuity, we only need to prove that (\ref{main estimate}) holds given that a solution $u$ of \eqref{main eqn} already exists. Note that $u$ satisfies
$$
\partial^{\alpha}_t u=\bar{a}^{ij}u_{x^ix^j}+\bar{f}, \quad \bar{f}=f+(a^{ij}-\bar{a}^{ij})u_{x^ix^j}.
$$
Hence by the assumption,
\begin{align*}
\|u\|_{\bH^{\alpha,2}_{q,p}(T)}&\leq N_0\|(a^{ij}-\bar{a}^{ij})u_{x^ix^j}\|_{\bL_{q,p}(T)}+N_0\|f\|_{\bL_{q,p}(T)}\\
&\leq N_0 \sup_{t,x}|a^{ij}-\bar{a}^{ij}|\, \|u\|_{\bH^{\alpha,2}_{q,p}(T)}+N_0\|f\|_{\bL_{q,p}(T)}.
\end{align*}
Thus it is enough to take $\varepsilon_0=(2N_0)^{-1}$.

{\bf{Step 4}}. Let $a^{ij}=a^{ij}(x)$ depend only on $x$ and $f=f_0$. In this case one can repeat the classical perturbation arguments to prove the claims.
Below we give a detail for the sake of the completeness. As before we only need to prove that there exists a constant $N_0$ independent of $u$ such that \eqref{main estimate} holds given that $u$ is a solution.
By Steps 2 and 3, there exists $\varepsilon_0>0$ depending only on $p,q, \alpha, K$, and $T$ such that the theorem holds true if there exists any point $x_0\in \fR^d$ such that
\begin{equation}
           \label{conti}
\sup_x |a^{ij}(x)-a^{ij}(x_0)|\leq \varepsilon_0.
\end{equation}
Recall that $a^{ij}$ is uniformly continuous. Let $\delta_0<1$ be a constant depending on $\varepsilon_0$ such that
$$
|a^{ij}(x)-a^{ij}(y)|\leq \varepsilon_0/2, \quad \text{if}\quad |x-y|<4\delta_0.
$$
Choose a partition of unity $\phi_n, n=1,2,\cdots$ so that $\phi_n=\phi(\frac{x-x_n}{\delta_0})$ for some $x_n\in \fR^d$ and $\phi\in C_c^\infty(B_2(0))$ satisfying $0\leq \phi\leq 1$ and  $\phi=1$ for $|x|\leq 1$. Denote $\bar{\phi}_n=\phi(\frac{x-x_n}{2\delta_0})$. Then $\bar{\phi}_n=1$ on the support of $\phi_n$, and $u_n=\phi_n u$ satisfies
$$
\partial^{\alpha}_t u_n=a^{ij}_n (u_n)_{x^ix^j}+f_n
$$
where
$$
a^{ij}_n=\bar{\phi}_n a^{ij}(x)+(1-\bar{\phi}_n)a^{ij}(x_n),
$$
$$
f_n=f\phi_n-2a^{ij}u_{x^i}(\phi_n)_{x^j}-a^{ij}u(\phi_n)_{x^ix^j}.
$$
It is easy to check $(a^{ij}_n)$ satisfies \eqref{eqn elliptic} and \eqref{conti} with $x_n$ in place of $x_0$. By \cite[Lemma 6.7]{Krylov1999} and Step 3, if $p=q$ then for any $t\leq T$,
\begin{align}
\|u\|^q_{\bH^{0,2}_{q,p}(t)} &\sim \sum_{n=1}^{\infty} \|u\phi_n\|^q_{\bH^{0,2}_{q,p}(t)}\leq N_0 \sum_{n=1}^{\infty} \label{equiv rel} \|f_n\|^q_{\bL_{q,p}(t)}\\ \nonumber
&\leq N\sum_{n=1}^{\infty}\left[\|u_x(\phi_n)_x\|^q_{\bL_{q,p}(t)}+\|u(\phi_n)_{xx}\|^q_{\bL_{q,p}(t)}+\|f\phi_n\|^q_{\bL_{q,p}(t)}\right]\\
&\leq N \|u\|^q_{\bH^{0,1}_{q,p}(t)}+N\|f\|^q_{\bL_{q,p}(t)}. \label{equiv rel2}
\end{align}
The last inequality above is also from Lemma 6.7 of \cite{Krylov1999}. We emphasize that equivalence relation in \eqref{equiv rel} and inequality \eqref{equiv rel2} hold in general only if $p=q$ or only finite $\phi_n$ are non-zero functions.  Hence, if $q\neq p$ then  we take sufficiently large $R,M>0$ so that  $\sum_{n=1}^M \phi_n(x) =1$ on $B_R$ and vanishes for $|x|\geq 2R$, and the oscillation of $a^{ij}$ on the complement of $B_{R/2}$ is less then $\varepsilon_0/2$. Denote $\phi_0=1-\sum_{n=1}^M \phi_n$. Then one can repeat the above calculations and use the relation $\|u\|_{H^1_p}\leq \varepsilon \|u\|_{H^2_p}+N\|u\|_{L_p}$  to  conclude that for all $t\leq T$
\begin{align*}
\|u\|^q_{\bH^{0,2}_{q,p}(t)}&\leq N \|u\|^q_{\bL_{q,p}(t)}+N\|f\|^q_{\bL_{q,p}(t)}\\
&\leq N\int^t_0 \left[\int^s_0 (s-r)^{-1+\alpha}(\|\partial_t^\alpha u(r)\|_{H^2_p}+\|f(r)\|_{L_p})dr\right]^q ds + N \|f\|^q_{\bL_{q,p}(t)}\\
&\leq N \int^t_0 (t-s)^{-1+\alpha}\|u\|^q_{\bH^{0,2}_{q,p}(s)}ds +N \|f\|^q_{\bL_{q,p}(T)},
\end{align*}
where Theorem \ref{lem 9.21.11} (iv) is used in the second inequality.
Hence by a version of Gronwall's lemma (see e.g. \cite[Corollary 2]{YGD}) we get
\begin{equation}
           \label{11.04.2}
\|u\|^q_{\bH^{0,2}_{q,p}(T)}\leq N \|f\|^q_{\bL_{q,p}(T)}.
\end{equation}
From equation \eqref{main eqn}, we easily conclude that
$$
\|\partial^{\alpha}_tu\|_{\bL_{q,p}(T)}\leq N \|u\|_{\bH^{0,2}_{q,p}(T)}+\|f\|_{\bL_{q,p}(T)},
$$
and therefore \eqref{11.04.2} certainly leads to the a priori estimate.

{\bf{Step 5}}. Let $f_0=f$, and $a^{ij}$ be uniformly continuous  in $(t,x)$.
As before,  we only need to prove the a priori estimate \eqref{main estimate}.

 Let  $N_0$ be the constant from Step 4 so that a priori estimate \eqref{main estimate} holds whenever $a^{ij}$ are independent of $t$. Take $\varepsilon_0$ from Step 3 corresponding to this $N_0$.
We will apply Step 3 with $\bar{a}^{ij}=a^{ij}(t_0,x)$ for some $t_0$.

Take $\kappa>0$ so that
$$
|a^{ij}(t,x)-a^{ij}(s,x)|\leq \varepsilon_0/2, \quad \text{if}\quad |t-s|\leq 2\kappa.
$$
Also take an integer $N$ so that $T/N\leq \kappa$, and denote $\tilde T_i=iT/N$.
By Steps 3 and 4 applied with $\bar{a}^{ij}(t,x)=a^{ij}(0,x)$, \eqref{main estimate} holds with $T_1$ and $2N_0$ in place of $T$ and $N_0$ respectively.
Now we use use the induction.
Suppose that the a priori estimate \eqref{main estimate} holds for $\tilde T_k<T$  with $N_0$ independent of $u$.
This constant $N_0$ may depend on $k$.
Take $\tilde{u}$ from Lemma \ref{cor 9.21.10} corresponding to $\tilde{T}=\tilde T_k$. Denote
$$
\bar{u}(t,x)=(u-\tilde{u})(\tilde{T}+t,x),  \quad \bar{f}(t,x)=f(\tilde{T}+t,x).
 $$
Then one can easily check that $\bar{u}$ satisfies
$$
\partial^{\alpha}_t \bar{u}=a^{ij}(\tilde{T}+t,x)\bar{u}_{x^ix^j}+\bar{f}+(a^{ij}(\tilde{T}+t,x)-\delta^{ij})\tilde{u}_{x^ix^j}(\tilde{T}+t,x), \quad t\leq T-\tilde{T}.
$$
Due to Lemma \ref{restric zero}, $\bar{u}$ is contained in $\bH^{\alpha,n}_{q,p,0}(T)$.
Thus by the result of Steps 3 and 4 with $\bar{a}^{ij}=a^{ij}(\tilde T_k,x)$, we have
\begin{align*}
\|\bar{u}\|^q_{\bH^{\alpha,2}_{q,p}(T_1)}
&\leq (2N_0)^q\int^{\tilde T_{k+1}}_{\tilde T_k}\|f+(a^{ij}-\delta^{ij})\tilde{u}_{x^ix^j}\|^q_{L_p}dt\\
&\leq N\|\tilde{u}\|^q_{\bH^{0,2}_{q,p}(T)}+N\|f\|^q_{\bL_{q,p}(T)} \\
&\leq  N \|u\|^q_{\bH^{\alpha,2}_{q,p}(\tilde T_k)}+\|f\|^q_{\bL_{q,p}(T)}\leq N \|f\|^q_{\bL_{q,p}(T)},
\end{align*}
where the third inequality is due to \eqref{ext est} and the last inequality is from the assumption. Hence
\begin{align*}
\|u\|^q_{\bH^{\alpha,2}_{q,p}(\tilde T_{k+1})}
&\leq \|\tilde u\|^q_{\bH^{\alpha,2}_{q,p}(\tilde T_{k+1})} + \|\bar u\|^q_{\bH^{\alpha,2}_{q,p}(\tilde T_{1})}  \leq N \|f\|^q_{\bL_{q,p}(T)}.
\end{align*}
As the induction goes through, the a priori estimate \eqref{main estimate} is proved. We emphasize that for each $k$ the constant $N_0$ varies, however for each $\tilde T_k$ we use the result in Step 4 and therefore the choice of $\tilde T_{k+1}$ (or the difference $|\tilde T_{k+1}-\tilde T_k|$) does not depend on $k$, and therefore we can reach up to $T$ by finite such steps.

{\bf{Step 6}}.
Let $f_0=f$ and $a^{ij}=\sum_{k=1}^{\ell} a^{ij}_k(t,x) I_{(T_{k-1},T_k]}(t)$.
If $\ell=1$, \eqref{main estimate} comes directly from Step 5.
If $\ell>1$, we use the induction argument used in Step 5. The only difference is that to estimate the solution on $[T_{k}, T_{k+1})$ we  use the result of Step 5, in place of the result in Step 4.

{\bf{Step 7}}. The general non-linear case. We modify the proof of Theorem 5.1 of \cite{Krylov1999}.
  For each $u\in \bH^{\alpha,2}_{q,p,0}(T)$  consider the equation
$$
  \partial^{\alpha}_t v=a^{ij}+f(u), \quad t\leq T.
$$
 By the above results, this equation has a unique solution
$v\in \bH^{\alpha,2}_{q,p,0}(T)$.
By denoting $v=\cR u$  we can define an operator
$\cR : \,\,\bH^{\alpha,2}_{q,p,0}(T) \to \bH^{\alpha,2}_{q,p,0}(T)$. By the results for the linear case, for each $t\leq T$,
\begin{align*}
\|\cR u-\cR v\|^q_{\bH^{\alpha,2}_{q,p}(t)}
&\leq N \|f(u)-f(v)\|^q_{\bL_{q,p}(t)}\\
&\leq N\varepsilon^q\|u-v\|^q_{\bH^{0,2}_{q,p}(t)}+NK^q_{\varepsilon}\|u-v\|^q_{\bL_{q,p}(t)}\\
&\leq N_0\varepsilon^q\|u-v\|^q_{\bH^{\alpha,2}_{q,p}(t)}+N_1\int^t_0 (t-s)^{-1+\alpha}\|u-v\|^q_{\bH^{\alpha,2}_{q,p}(s)}\,ds,
\end{align*}
where $N_1$ depends also on $\varepsilon$, and Theorem \ref{lem 9.21.11} (iv) is used in the last inequality.
Next, we fix $\varepsilon$ so that  $\theta:=N_0\varepsilon^q<1/4$. Then repeating the above inequality and using the identity
\begin{align*}
&\int^t_0 (t-s_1)^{-1+\alpha}\int^{s_1}_0 (s_1-s_2)^{-1+\alpha}\cdots \int^{s_{n-1}}_0 (s_{n-1}-s_n)^{-1+\alpha} ds_n \cdots ds_1\\
&=\frac{\left\{\Gamma(\alpha)\right\}^{n}}{\Gamma(n\alpha+1)} t^{n\alpha},
\end{align*}
we get
\begin{align*}
&\|\cR^m u-\cR^m v\|^q_{\bH^{\alpha,2}_{q,p}(T)}
\\
&\leq \sum_{k=0}^m \begin{pmatrix} m\\ k\end{pmatrix}
\theta^{m-k} (T^{\alpha}N_1)^k \frac{\left\{\Gamma(\alpha)\right\}^{k}}{\Gamma(k\alpha+1)}  \, \|u-v\|^q_{\bH^{\alpha,2}_{q,p}(T)}\\
&\leq 2^m \theta^m \left[\max_k \left( \frac{(\theta^{-1}T^{\alpha}N_1\Gamma(\alpha))^k}{\Gamma(k\alpha+1)}\right)\right]\, \|u-v\|^q_{\bH^{\alpha,2}_{q,p}(T)}\\
&\leq
\frac{1}{2^m} N_2 \|u-v\|^q_{\bH^{\alpha,2}_{q,p}(T)}.
\end{align*}
For the second inequality above we use $\sum_{k=0}^m \begin{pmatrix} m\\ k\end{pmatrix}=2^m$. It follows that if $m$ is sufficiently large then
$\cR^m$ is a contraction in $\bH^{\alpha,2}_{q,p,0}(T)$, and this yields all the claims.
The theorem is proved. \qed

\mysection{Kernels $p$ and $q$}
                    \label{p q kernel sec}

\subsection{The kernel $p(t,x)$}

In this subsection, we prove Lemma \ref{p exists}(i).
In other words, we introduce a  kernel $p(t,x)$ which is integrable with respect to $x$  and satisfies
\begin{align}\label{appendix:motivation}
\mathcal{F}\{p(t,\cdot)\}(\xi)=E_{\alpha}(-t^{\alpha}|\xi|^{2}).
\end{align}

 Let $\Gamma(z)$ denote the gamma function which can be defined (see \cite[Section 1.1]{Gamma}) for $z\in\mathbb{C}\setminus\{0,-1,-2,\ldots\}$   as
$$
\Gamma(z)=\lim_{n\rightarrow\infty}\frac{n! n^{z}}{z(z+1)\cdots(z+n)}.
$$
Note that $\Gamma(z)$ is a meromorphic function with simple poles at the nonpositive integers. From the definition, for $z\in\mathbb{C}\setminus\{0,-1,-2,\ldots\}$,
\begin{align}
				\label{appendix:gamma1}
z\Gamma(z)=\Gamma(z+1).
\end{align}
By \eqref{appendix:gamma1}, one can easily check that for $k=0,1,2,\ldots$,
\begin{align}
				\label{appendix:gamma-pole}
				\textnormal{Res}_{z=-k}\Gamma(z)=\lim_{z\rightarrow -k}(z+k)\Gamma(z) = \frac{(-1)^{k}}{k!},
\end{align}
where $\textnormal{Res}_{z=-k}\Gamma(z)$ denotes the residue of $\Gamma(z)$ at $z=-k$. It is also well-known (see. e.g. \cite[Theorem 1.1.4]{Gamma}) that if  $\Re[z]>0$ and $\Re[\omega]>0$ then
\begin{align}
				\label{appendix:Beta}
	\Gamma(z)=\int_{0}^{\infty}t^{z-1}e^{-t}dt, \quad 			 \int_{0}^{1}t^{z-1}(1-t)^{\omega-1}dt = \frac{\Gamma(z)\Gamma(\omega)}{\Gamma(z+\omega)}.
\end{align}
 Using Stirling's approximation (see \cite[Corollary 1.4.3]{Gamma} or \cite[(1.2.3)]{kilbas2004h})
\begin{align}
                \label{stirling}
\Gamma(z)\sim \sqrt{2\pi}e^{(z-\frac{1}{2})\log{z}}e^{-z},\quad |z|\rightarrow\infty,
\end{align}
 one can easily show that for fixed $a\in\mathbf{R}$, (\cite[(1.2.2)]{kilbas2004h})
\begin{align}\label{appendix:gamma decay}
|\Gamma(a+\textnormal{i}b)|\sim \sqrt{2\pi}|b|^{a-\frac{1}{2}}e^{-a-\frac{\pi}{2}|b|},\quad |b|\rightarrow\infty.
\end{align}

Let $m,n,\mu,\nu$ be fixed integers satisfying $0\leq m\leq \mu$, $0\leq n\leq \nu$.
Assume that the complex parameters $c_{1},\ldots, c_{\nu}$, $d_{1},\ldots,d_{\mu}$ and positive real parameters $\gamma_{1},\ldots,\gamma_{\nu}$, $\delta_{1},\ldots,\delta_{\mu}$ are given so that $P_1 \cap P_2 = \emptyset$ where
$$
P_{1} := \left\{-\frac{d_{j}+k}{\delta_{j}}\in\mathbb{C}:j\in\{1,\ldots,m\},\ k=0,1,2,\ldots\right\}
$$
$$
P_{2} := \left\{\frac{1-c_{j}+k}{\gamma_{j}}\in\mathbb{C}:j\in\{1,\ldots,n\},\ k=0,1,2,\ldots\right\}.
$$
If either $m=0$ or $n=0$, then by the definition $P_1 \cap P_2= \emptyset$.
For the above parameters,  the Fox H-function $H(r)$ ($r>0$) is defined as
\begin{align}
				\label{H function}
H(r)&:=\textnormal{H}_{\nu\mu}^{mn}\left[r\ \Big|\begin{array}{cccccc}
(c_{1},\gamma_{1}) & \cdots & (c_n,\gamma_{n}) & (c_{n+1},\gamma_{n+1}) & \cdots (c_{\nu},\gamma_{\nu}) \\
(d_{1},\delta_{1}) & \cdots & (d_{m},\delta_{m}) & (d_{m+1},\delta_{m+1}) & \cdots (d_{\mu},\delta_{\mu}).
\end{array}\right]\nonumber
\\ &:=\frac{1}{2\pi \textnormal{i}}\int_{L}\frac{\prod_{j=1}^{m}\Gamma(d_{j}+\delta_{j}z)\prod_{j=1}^{n}\Gamma(1-c_{j}-\gamma_{j}z)}{\prod_{j=m+1}^{\mu}\Gamma(1-d_{j}-\delta_{j}z)\prod_{j=n+1}^{\nu}\Gamma(c_{j}+\gamma_{j}z)}r^{-z}dz,
\end{align}
where the contour $L$ is chosen appropriately depending on the parameters. Some special cases needed in our setting are specified below.
In this article, we  additionally assume that parameters $c_{1},\ldots,c_{\nu}$ and $d_{1},\ldots,d_{\mu}$ are real and
\begin{align}
			\label{assumption:analytic}
\sum_{j=1}^{\mu}\delta_{j}-\sum_{i=1}^{\nu}\gamma_{i}>0,\quad \sum_{i=1}^{n}\gamma_{i}-\sum_{i=n+1}^{\nu}\gamma_{i}+
\sum_{j=1}^{m}\delta_{j}-\sum_{j=m+1}^{\mu}\delta_{j}>0.
\end{align}
Under \eqref{assumption:analytic}, we can choose the contour $L$ of two different types.
Hankel contour $L_{h}$ is a  loop  starting at the point $-\infty+\textnormal{i}\rho_{1}$ and ending at the point $-\infty+\textnormal{i}\rho_{2}$ where $\rho_{1}< 0 < \rho_{2}$,
 which encircles all the poles of $P_{1}$ once in the positive direction but none of the poles of $P_{2}$.
Bromwich contour $L_{v}$ is a vertical contour, which go from $\gamma_{0}-\textnormal{i}\infty$ to $\gamma_{0}+\textnormal{i}\infty$ where $\gamma_{0}\in\mathbf{R}$ and  leaves all the poles of $P_{1}$ to the right and all poles of $P_{2}$ to the left.
Braaksma \cite{braaksma1936asymptotic} showed that the contour integral \eqref{H function} makes sense along  $L_{h}$ and $L_{v}$ for $r\in(0,\infty)$, and two integrals along $L_{h}$ and $L_{v}$  coincide  (see \cite[Section 1.2]{kilbas2004h}). Furthermore, if $\sum_{j=1}^{\mu}\delta_{j}-\sum_{i=1}^{\nu}\gamma_{i}>0$  then $H(r)$ is an analytic function on $(0,\infty)$ and can be represented (see \cite[Theorem 1.2]{kilbas2004h})
as
\begin{align}
				\label{appendix:residue}
H(r)=\sum_{i=1}^{m}\sum_{k=0}^{\infty}\textnormal{Res}_{z=\hat{d}_{ik}}
\left[\frac{\prod_{j=1}^{m}\Gamma(d_{j}+\delta_{j}z)\prod_{j=1}^{n}\Gamma(1-c_{j}-\gamma_{j}z)}
{\prod_{j=m+1}^{\mu}\Gamma(1-d_{j}-\delta_{j}z)\prod_{j=n+1}^{\nu}\Gamma(c_{j}+\gamma_{j}z)}r^{-z}\right],
\end{align}
where $\hat{d}_{ik}:=-(d_{i}+k)/\delta_{i}\in P_{1}$ are poles of the integrand in the contour integral, $i\in\{1,\ldots,m\}$, and $k=0,1,2,\cdots$.
Note that on the negative real half-axis the Mittag-Leffler function
$$
E_{\alpha}(z)=\sum_{k=0}^{\infty}\frac{z^{k}}{\Gamma(\alpha k +1)},\quad \alpha>0,
$$
can be written as
\begin{align}
                    \label{mit ref}
E_{\alpha}(-x)=\textnormal{H}_{12}^{11}\left[x\ \Big|\begin{array}{cc} (0,1) \\(0,1) &(0,\alpha)\end{array}\right] \quad \quad (x >0).
\end{align}
Indeed, by \eqref{appendix:residue} and \eqref{appendix:gamma-pole}
\begin{align*}
				&\textnormal{H}_{12}^{11}\left[x\ \Big|\begin{array}{cc} (0,1) \\ (0,1)&(0,\alpha)\end{array}\right]=\sum_{k=0}^{\infty}\textnormal{Res}_{z=-k}\left[\frac{\Gamma(z)\Gamma(1-z)}{\Gamma(1-\alpha z)}x^{-z}\right]\nonumber \\
				&=\sum_{k=0}^{\infty}\frac{(-1)^{k}\Gamma(1+k)}{k!\Gamma(\alpha k+1)}x^{k}
				=\sum_{k=0}^{\infty}\frac{(-x)^{k}}{\Gamma(\alpha k+1)}=E_{\alpha}(-x).
\end{align*}

We introduce some notion to introduce the asymptotic behavior of $H(r)$. Let $j_{s}$ ($1\leq j_{s} \leq m$) be the number such that
$$
\rho_{s}:=\frac{d_{j_{s}}}{\delta_{j_{s}}}=\min\left[\frac{d_{j}}{\delta_{j}}\right] \quad \quad (\min \emptyset:=\infty)
$$
where the minimum are taken over all $\frac{d_{j}}{\delta_{j}}$ so that $\hat{d}_{j0}$ is a simple pole.
Similarly let $j_{c}$ ($1\leq j_{c}\leq m$) be the number such that
$$
\rho_{c}:=\frac{d_{j_{c}}}{\delta_{j_{c}}}=\min \left[\frac{d_{j}}{\delta_{j}}\right]
$$
where the minimum are taken over all $\frac{d_{j}}{\delta_{j}}$ so that $\hat{d}_{j0}$ is a pole with order $n_c\geq 2$.
Here $n_c$ denotes the smallest number of the orders for non-simple poles.

The following result can be  proved  on the basis of \eqref{appendix:residue}.
See \cite[Corollary 1.12.1]{kilbas2004h} for the proof.

\begin{thm}
				\label{appendix:kernel estimate2}
		
(i) If $\rho_{s} < \rho_{c}$, then for $r \leq 1$
$$
|H(r)|\leq N r^{\rho_{s}};
$$

(ii) if $\rho_{s} \geq \rho_{c}$, then for $r\leq 1$
$$
|H(r)|\leq N r^{\rho_{c}}|\ln r|^{n_{c}-1}.
$$
				
\end{thm}

An upper bound of $H(r)$ on $[1,\infty)$ is also well-known if $n=0$ and $m=\mu$ in \eqref{H function}.
See \cite[Corollary 1.10.2]{kilbas2004h} and \cite[(2.2.2)]{kilbas2004h}.

\begin{thm}
				\label{appendix:kernel estimate1}
Suppose that $n=0$ and $m=\mu$ in \eqref{H function}. Then for $r\geq 1$,
\begin{align}\label{appendix:3}
|H(r)|\leq N r^{(\Lambda+1/2)\omega^{-1}}\exp\{-\omega\eta^{-1/\omega}r^{1/\omega}\},
\end{align}
where
$$
\Lambda:=\sum_{j=1}^{\mu}d_{j}-\sum_{i=1}^{\nu}c_{i}+\frac{\nu-\mu}{2},\quad \omega:=\sum_{j=1}^{\mu}\delta_{j}-\sum_{i=1}^{\nu}\gamma_{i},\quad \eta:=\prod_{j=1}^{\mu}\delta_{j}^{\delta_{j}}\prod_{i=1}^{\nu}\gamma_{i}^{-\gamma_{i}}.
$$
\end{thm}

\begin{thm}
Suppose that $n=0$ amd $m=\mu$ in \eqref{H function}. Then
\begin{align}
				\label{appendix:diff}
\frac{d}{dr}H(r)=-r^{-1}\textnormal{H}_{\nu+1\mu+1}^{\mu+1 \ 0}\left[r\ \Big|\begin{array}{cccc}(c_{1},\gamma_{1}) & \cdots & (c_{\nu},\gamma_{\nu}) & (0,1) \\ (d_{1},\delta_{1}) & \cdots & (d_{\mu},\delta_{\mu}) & (1,1)\end{array} \right].
\end{align}
\end{thm}

\vspace{3mm}

Now we define $p(t,x)$ so that \eqref{appendix:motivation} holds. Fix $\alpha \in (0,2)$ and let
\begin{align}
				\label{p:H function}
p(t,x)&:=\pi^{-\frac{d}{2}}|x|^{-d}\textnormal{H}_{12}^{20}\left[ \frac{1}{4}t^{-\alpha}|x|^{2}\ \Big|\begin{array}{cc} (1,\alpha) \\ (\frac{d}{2},1) & (1,1)\end{array}\right] \nonumber
\\ &=\pi^{-\frac{d}{2}}|x|^{-d} \frac{1}{2\pi\textnormal{i}}\int_{L_{v}}\frac{\Gamma(\frac{d}{2}+ z)\Gamma(1+z)}{\Gamma(1+\alpha z)}\left(\frac{1}{4}t^{-\alpha}|x|^{2}\right)^{-z}dz.
\end{align}
One can easily check that \eqref{assumption:analytic} is satisfied.
Hence we can take the Bromwich contour $L=L_{v}$, which runs along from $-\gamma-\textnormal{i}\infty$ to $-\gamma+\textnormal{i}\infty$, i.e.
$$
L_{v}:=\{z\in\mathbb{C}:\Re[z]=-\gamma\}
$$
where
$$0<\gamma<\min\{1,\frac{d-1}{4},\frac{1}{\alpha}\} \quad \text{if}\,\, d\geq 2\quad  \text{and} \quad  0<\gamma<\min\{\frac{1}{2},\frac{1}{\alpha}\} \quad \text{if}\,\, d=1.
$$
 Note that
\begin{align*}
\min\{\rho_{s},\rho_{c}\}=
\begin{cases}\rho_{s}=1 & \text{if} \quad \hbox{$d\geq 3$} \\
\rho_{c}=1,\,n_{c}=2 & \text{if} \quad \hbox{$d=2$}
\\ \rho_{s}=\frac{1}{2}& \text{if} \quad \hbox{$d=1$.}
\end{cases}
\end{align*}
By Theorem \ref{appendix:kernel estimate2} and Theorem \ref{appendix:kernel estimate1},
$$p(t,\cdot)\in L_{1}(\mathbf{R}^{d}), \quad \forall\, t>0.
 $$
 For the  asymptotic behavior  of the integrand in \eqref{p:H function} we use Stirling's approximation.
Write $z=-\gamma+\textnormal{i}\tau$, $\tau\in(-\infty,\infty)$.
Then by \eqref{appendix:gamma decay},  for $\rho\in(0,\infty)$, $t\in(0,\infty)$, and large $|\tau|$,
\begin{align}\label{appendix:decay}
\left|\frac{\Gamma(\frac{d}{2}+z)\Gamma(1+z)}{\Gamma(1+\alpha z)}\left(\frac{1}{4}t^{-\alpha}\rho^{2}\right)^{-z}\right|\leq N\left(\frac{1}{4}t^{-\alpha}\rho^{2}\right)^{\gamma}|\tau|^{c_{1}}e^{-c_{2}|\tau|},
\end{align}
where
$$
c_{1}:=-\gamma(2-\alpha)+\frac{d-1}{2},\quad c_{2}:=\frac{\pi}{2}(2-\alpha).
$$
Therefore,
\begin{align}\label{appendix:bound}
\sup_{z\in L}\left|\frac{\Gamma(\frac{d}{2}+z)\Gamma(1+z)}{\Gamma(1+\alpha z)}\left(\frac{1}{4}t^{-\alpha}\rho^{2}\right)^{-z}\right|\leq N\left(\frac{1}{4}t^{-\alpha}\rho^{2}\right)^{\gamma},
\end{align}
  because for fixed $a\notin\{0,-1,-2,-3,\ldots\}$ the mapping $b\mapsto\Gamma(a+\textnormal{i}b)$ is a continuous function and does not vanish on $\mathbf{R}$.  Hence
\begin{align}
				\label{appendix:1}
&\left|\int_{L}\frac{\Gamma(\frac{d}{2}+z)\Gamma(1+z)}{\Gamma(1+\alpha z)}\left(\frac{1}{4}t^{-\alpha}\rho^{2}\right)^{-z}dz\right| \nonumber\\
&\quad\quad\quad\leq \int_{-\infty}^{\infty}\left|\frac{\Gamma(\frac{d}{2}-\gamma+\textnormal{i}\tau)\Gamma(1-\gamma+\textnormal{i}\tau))}{\Gamma(1-\alpha \gamma+\textnormal{i}\alpha\tau)}\left(\frac{1}{4}t^{-\alpha}\rho^{2}\right)^{\gamma-\textnormal{i}\tau}\right|d\tau\nonumber \\
&\quad\quad\quad\leq Nt^{-\alpha\gamma}\rho^{2\gamma}\left\{N+\int_{|\tau|\geq 1}|\tau|^{c_{1}}e^{-c_{2}|\tau|}d\tau\right\}\leq Nt^{-\alpha\gamma}\rho^{2\gamma}.
\end{align}
We remark that \eqref{appendix:decay}, \eqref{appendix:bound}, and \eqref{appendix:1} hold for any $0<\gamma< \min\{1,\frac{d}{2},\frac{1}{\alpha}\} $.

Now we prove \eqref{appendix:motivation}. First assume $d\geq 2$. By the formula for the Fourier transform of a radial function (see  \cite[Theorem IV.3.3]{stein1971introduction}),
we have
\begin{align*}
\mathcal{F}\{p(t,\cdot)\}(\xi)&=\frac{2^{d/2}}{|\xi|^{\frac{d}{2}-1}}\int_{0}^{\infty}\rho^{-\frac{d}{2}}\textnormal{H}_{12}^{20}\left[\frac{1}{4}t^{-\alpha}\rho^{2}\ \Big|\begin{array}{cc} (1,\alpha)\\ (\frac{d}{2},1)&(1,1)\end{array}\right]J_{\frac{d}{2}-1}(|\xi|\rho)d\rho,
\end{align*}
where $J_{\frac{d}{2}-1}$ is the Bessel function of the first kind of order $\frac{d}{2}-1$, i.e. for $r\in[0,\infty)$,
$$
J_{\frac{d}{2}-1}(r)=\sum_{k=0}^{\infty}\frac{(-1)^{k}}{k!\Gamma(k+\frac{d}{2})}\left(\frac{r}{2}\right)^{2k-1+d/2}.
$$
 It is well-known    (e.g. \cite[(2.6.3)]{kilbas2004h}) that if  $m>-1$ then
$$
J_{m}(t)= \begin{cases}O(t^{m}), & \hbox{$t\rightarrow 0+$} \\ O(t^{-1/2}), & \hbox{$t\rightarrow\infty$}. \end{cases}
$$
Thus
\begin{align}\label{appendix:2}
\int_{0}^{\infty}|\rho^{-\frac{d}{2}+2\gamma}J_{\frac{d}{2}-1}(|\xi|\rho)|d\rho\leq \int_{0}^{1}\rho^{2\gamma-1}d\rho+\int_{1}^{\infty}\rho^{-\frac{d}{2}+2\gamma-\frac{1}{2}}d\rho < \infty
\end{align}
since $0<\gamma<\min\{1,\frac{d-1}{4},\frac{1}{\alpha}\}$.
By the definition of the Fox H-function,
\begin{align}
				\label{appendix:contour integral}
&\int_{0}^{\infty}\rho^{-\frac{d}{2}}\textnormal{H}_{12}^{20}\left[\frac{1}{4}t^{-\alpha}\rho^{2}\ \Big| \begin{array}{cc} (1,\alpha) \\ (\frac{d}{2},1) & (1,1) \end{array}\right]J_{\frac{d}{2}-1}(|\xi|\rho)d\rho \nonumber\\
&=\frac{1}{2\pi\textnormal{i}}\int_{0}^{\infty}\rho^{-\frac{d}{2}}\left[\int_{L}\frac{\Gamma(\frac{d}{2}+z)\Gamma(1+z)}{\Gamma(1+\alpha z)}\left(\frac{1}{4}t^{-\alpha}\rho^{2}\right)^{-z}dz\right]J_{\frac{d}{2}-1}(|\xi|\rho)d\rho.
\end{align}
Combining \eqref{appendix:1} and \eqref{appendix:2}, we have
$$
\int_{0}^{\infty}\int_{L}\left|\rho^{-\frac{d}{2}}\frac{\Gamma(\frac{d}{2}+z)\Gamma(1+z)}{\Gamma(1+\alpha z)}
\left(\frac{1}{4}t^{-\alpha}\rho^{2}\right)^{-z}
J_{\frac{d}{2}-1}(|\xi|\rho)\right||dz|d\rho<\infty.
$$
Thus we can apply Fubini's theorem to \eqref{appendix:contour integral}.
Furthermore, since
$$
-\frac{d}{2}<\Re[-\frac{d}{2}-2z]<-\frac{1}{2} \quad \quad \forall z \in L,
$$
by using the formula \cite[(2.6.4)]{kilbas2004h}
\begin{align*}
\int_{0}^{\infty}\rho^{-\frac{d}{2}-2z}J_{\frac{d}{2}-1}(|\xi|\rho)d\rho = 2^{-\frac{d}{2}-2z}|\xi|^{\frac{d}{2}-1+2z}\frac{\Gamma(-z)}{\Gamma(\frac{d}{2}+z)},
\end{align*}
we get
\begin{align*}
&\int_{0}^{\infty}\rho^{-\frac{d}{2}}\textnormal{H}_{12}^{20}\left[\frac{1}{4}t^{-\alpha}\rho^{2}\ \Big|\begin{array}{cc} (1,\alpha)\\ (\frac{d}{2},1)&(1,1)\end{array}\right]J_{\frac{d}{2}-1}(|\xi|\rho)d\rho \\
&=\frac{1}{2\pi\textnormal{i}}\int_{L}\left[\int_{0}^{\infty}\rho^{-\frac{d}{2}-2z}J_{\frac{d}{2}-1}(|\xi|\rho)d\rho\right]\frac{\Gamma(\frac{d}{2}+z)\Gamma(1+z)}{\Gamma(1+\alpha z)}\left(\frac{1}{4}t^{-\alpha}\right)^{-z}dz \\
&=\frac{|\xi|^{\frac{d}{2}-1}}{2^{d/2}}\frac{1}{2\pi\textnormal{i}}\int_{L}\frac{\Gamma(\frac{d}{2}+z)\Gamma(1+z)\Gamma(-z)}{\Gamma(1+\alpha z)\Gamma(\frac{d}{2}+z)}(t^{-\alpha}|\xi|^{-2})^{-z}dz\\
&=\frac{|\xi|^{\frac{d}{2}-1}}{2^{d/2}}\frac{1}{2\pi\textnormal{i}}\int_{-L}\frac{\Gamma(1-z)\Gamma(z)}{\Gamma(1-\alpha z)}(t^{\alpha}|\xi|^{2})^{-z}dz\\
&=\frac{|\xi|^{\frac{d}{2}-1}}{2^{d/2}}\textnormal{H}_{12}^{11}\left[t^{\alpha}|\xi|^{2}\ \Big|\begin{array}{cc} (0,1) \\ (0,1) & (0,\alpha)\end{array}\right]=\frac{|\xi|^{\frac{d}{2}-1}}{2^{d/2}}E_{\alpha}(-t^{\alpha}|\xi|^{2})
\end{align*}
where the last equality is due to \eqref{mit ref}.
Therefore  \eqref{appendix:motivation} is proved for $d\geq 2$.

Next let $d=1$.
By \eqref{16.5},
$$
\mathcal{L}\left[E_{\alpha}(-t^{\alpha}|\xi|^{2})\right])(s)=\int_{0}^{\infty}e^{-st}E_{\alpha}(-t^{\alpha}|\xi|^{2})dt=\frac{s^{\alpha-1}}{s^{\alpha}+|\xi|^{2}}.
$$
Thus it suffices to prove
\begin{align}
				\label{appendix:1-dim case}
\mathcal{L}\left[\mathcal{F}\left\{p(t,\cdot)\right\}\right](s)=\int_{0}^{\infty}e^{-st}\mathcal{F}\left\{p(t,\cdot)\right\}(\xi)dt=\frac{s^{\alpha-1}}{s^{\alpha}+|\xi|^{2}}.
\end{align}
By Theorems \ref{appendix:kernel estimate2} and \ref{appendix:kernel estimate1} it holds that for each $s>0$,
\begin{align*}
 \int_{0}^{\infty}\int_{\mathbf{R}}e^{-st}|p(t,x)|dxdt<\infty.
\end{align*}
Hence by Fubini's theorem,
$\mathcal{L}\left[\mathcal{F}\left\{p(t,\cdot)\right\}\right](s)
=\mathcal{F}\left\{\mathcal{L}\left[p(\cdot,x)\right](s)\right\}$.
Due to \eqref{appendix:Beta} and \eqref{appendix:1},
\begin{align*}
&\mathcal{L}\left[p(\cdot,x)\right](s)=\int_{0}^{\infty}e^{-st}p(t,x)dt \\
			&= \pi^{-\frac{1}{2}}|x|^{-1}\int_{0}^{\infty}e^{-st}\left[\frac{1}{2\pi\textnormal{i}}\int_{L}\frac{\Gamma(\frac{1}{2}+z)\Gamma(1+z)}{\Gamma(1+\alpha z)}\left(\frac{1}{4}t^{-\alpha}|x|^{2}\right)^{-z}dz\right]dt \\
			 &=\frac{\pi^{-\frac{1}{2}}|x|^{-1}}{2\pi\textnormal{i}}\int_{L}\frac{\Gamma(\frac{1}{2}+z)\Gamma(1+z)}{\Gamma(1+\alpha z)}\left(\frac{1}{4}|x|^{2}\right)^{-z}\left[\int_{0}^{\infty}e^{-st}t^{\alpha z}dt\right]dz\\
			 &=\frac{\pi^{-\frac{1}{2}}|x|^{-1}s^{-1}}{2\pi\textnormal{i}}\int_{L}\Gamma(\frac{1}{2}+z)\Gamma(1+z)\left(\frac{1}{4}s^{\alpha}|x|^{2}\right)^{-z}dz.
\end{align*}
Furthermore by \cite[(2.9.19)]{kilbas2004h},
\begin{align*}
&\frac{1}{2\pi \textnormal{i}}\int_{L}\Gamma(\frac{1}{2}+z)\Gamma(1+z)\left(\frac{1}{4}s^{\alpha}|x|^{2}\right)^{-z}dz \\
&= \textnormal{H}_{02}^{20}\left[\frac{1}{4}s^{\alpha}|x|^{2}\Big|\begin{array}{cc} \\ (\frac{1}{2},1) & (1,1) \end{array}\right]
= 2\left(\frac{s^{\alpha/2}|x|}{2}\right)^{3/2}K_{1/2}(s^{\alpha/2}|x|),
\end{align*}
where $K_{\eta}(z)$ is called the modified Bessel function of the second kind\footnote{It has also been called the modified Bessel function of the third kind.} or the Macdonald function
and it satisfies (see \cite[9.7.2]{Bessel})
$$
K_{1/2}(z)=\sqrt{\frac{\pi}{2}}z^{-1/2}e^{-z},\quad |\textnormal{arg}z|<\frac{3\pi}{2}.
$$
Hence we obtain
$$
\mathcal{L}\left[p(\cdot,x)\right](s)=\frac{s^{\alpha/2-1}}{2}\exp\{-s^{\alpha/2}|x|\},
$$
which obviously implies that
\begin{align*}
\mathcal{F}\left\{\mathcal{L}\left[p(\cdot,x)\right](s)\right\}(\xi)& =\frac{s^{\alpha/2-1}}{2}\int_{-\infty}^{\infty}e^{-\textnormal{i}x\xi}\exp\{-s^{\alpha/2}|x|\}dx \\
				& = \frac{s^{\alpha/2-1}}{2}\cdot \frac{2s^{\alpha/2}}{s^{\alpha}+|\xi|^{2}} = \frac{s^{\alpha-1}}{s^{\alpha}+|\xi|^{2}}.
\end{align*}
Therefore \eqref{appendix:1-dim case} is proved, and  \eqref{appendix:motivation} holds.

\subsection{Representation of $q(t,x)$ and $K(t,x)$}
Let $\alpha \in (0,2)$ and recall
\begin{align*}
p(t,x):=\pi^{-\frac{d}{2}}|x|^{-d}\textnormal{H}_{12}^{20}\left[ \frac{1}{4}t^{-\alpha}|x|^{2}\ \Big|\begin{array}{cc} (1,\alpha)
 \\ (\frac{d}{2},1) & (1,1)\end{array}\right]. \nonumber
\end{align*}
By Theorems \ref{appendix:kernel estimate1} and \eqref{appendix:diff}, if $t\neq 0$ and $x\neq 0$ then $p(t,x)$ is differentiable in $t$ and $\lim_{t \to 0+} p(t,x)=0$. Thus we can define
$$
q(t,x):=\begin{cases}I_{t}^{\alpha-1}p(t,x), & \hbox{$\alpha\in(1,2)$} \\
D_{t}^{1-\alpha}p(t,x), &  \hbox{$\alpha\in(0,1)$} \end{cases}, \quad \quad K(t,x):=\frac{\partial}{\partial t}p(t,x).
$$
In this subsection, we derive the following representations:
\begin{align}\label{appendix:q}
q(t,x)=\pi^{-\frac{d}{2}}|x|^{-d}t^{\alpha-1}\textnormal{H}_{12}^{20}\left[\frac{1}{4}t^{-\alpha}|x|^{2}\ \Big|\begin{array}{cc} (\alpha,\alpha) \\ (\frac{d}{2},1) & (1,1)\end{array}\right]
\end{align}
and
\begin{align}
                    \label{K equal}
K(t,x)=\pi^{-\frac{d}{2}}|x|^{-d}t^{-1}\textnormal{H}_{12}^{20}\left[\frac{1}{4}t^{-\alpha}|x|^{2}\ \begin{array}{cc}(0,\alpha) \\ (\frac{d}{2},1) & (1,1)\end{array}\right].
\end{align}

We consider the Bromwich contour
$$
L_{v}:=\{z\in\mathbb{C}:\Re[z]=-\gamma\}, \quad \quad 0<\gamma<\min\{1,\frac{d}{2},\frac{1}{\alpha}\}.
$$
Let $\beta>0$.
By \eqref{appendix:1},
\begin{align*}
&\int_{0}^{t}(t-s)^{\beta-1}\left(\int_{L}\left|\frac{\Gamma(\frac{d}{2}+z)\Gamma(1+z)}{\Gamma(1+\alpha z)}\left(\frac{1}{4}s^{-\alpha}|x|^{2}\right)^{-z}\right||dz|\right)ds\\
&\leq N|x|^{2\gamma} \int^t_0(t-s)^{\beta-1}(t-s)^{-\alpha \gamma}ds <\infty.
\end{align*}
Thus by Fubini's theorem and \eqref{appendix:Beta},
\begin{align}
                 \notag
&I_{t}^{\beta}p(t,x)=\frac{1}{\Gamma(\beta)}\int_{0}^{t}(t-s)^{\beta-1}p(s,x)ds \\
                    \notag
&=\frac{\pi^{-\frac{d}{2}}|x|^{-d}}{\Gamma(\beta)}\int_{0}^{t}(t-s)^{\beta-1}\left[\frac{1}{2\pi\textnormal{i}}\int_{L}\frac{\Gamma(\frac{d}{2}+z)\Gamma(1+z)}{\Gamma(1+\alpha z)}\left(\frac{1}{4}s^{-\alpha}|x|^{2}\right)^{-z}dz\right]ds \\
                    \notag
&=\frac{\pi^{-\frac{d}{2}}|x|^{-d}}{2\pi\textnormal{i}}\int_{L}\frac{\Gamma(\frac{d}{2}+z)\Gamma(1+z)}{\Gamma(1+\alpha z)}\left[\frac{1}{\Gamma(\beta)}\int_{0}^{t}(t-s)^{\beta -1}s^{\alpha z}ds\right]\left(\frac{1}{4}|x|^{2}\right)^{-z}dz \\
                    \label{integ-beta-p}
&=\pi^{-\frac{d}{2}}|x|^{-d}t^{\beta}\frac{1}{2\pi\textnormal{i}}\int_{L}\frac{\Gamma(\frac{d}{2}+z)\Gamma(1+z)}{\Gamma(1+\beta+\alpha z)}\left(\frac{1}{4}t^{-\alpha}|x|^{2}\right)^{-z}dz \\
                    \notag
&=\pi^{-\frac{d}{2}}|x|^{-d}t^{\beta}\textnormal{H}_{12}^{20}\left[\frac{1}{4}t^{-\alpha}|x|^{2}\ \Big|\begin{array}{cc} (1+\beta,\alpha) \\ (\frac{d}{2},1) & (1,1)\end{array}\right].
\end{align}
Therefore, \eqref{appendix:q} is proved for $\alpha\in(1,2)$.

Next we differentiate kernels with respect to $t$.
Write $z=-\gamma+\textnormal{i}\tau$.
Following the method used to prove \eqref{appendix:decay} and \eqref{appendix:bound}, for any $\alpha \in (0,2)$ and $\beta \in [0,2)$  we have
\begin{align}
				\label{appendix:Fubini}
&\left|\frac{\Gamma(\frac{d}{2}+z)\Gamma(1+z)}{\Gamma(1+\beta+\alpha z)}\left(\frac{1}{4}|x|^{2}\right)^{-z}\right|\left|\frac{d}{dt}t^{\alpha z}\right|\nonumber\\
&\leq N(d,\gamma,\beta) t^{-1}\left(\frac{1}{4}t^{-\alpha}|x|^{2}\right)^{\gamma}\left[(\gamma^{2}+|\tau|^{2})^{1/2}(1_{|\tau|\geq 1}|\tau|^{c_{1}-\beta}e^{-c_{2}|\tau|}+1_{|\tau|\leq 1})\right].
\end{align}
Hence  the time derivative  of the integrand in \eqref{integ-beta-p} is   integrable in $z$ uniformly in a neighborhood of $t>0$, and thus by the dominated convergence theorem,
\begin{align*}
\frac{d}{dt}\textnormal{H}_{12}^{20}&\left[\frac{1}{4}t^{-\alpha}|x|^{2}\ \Big|\begin{array}{cc} (1+\alpha,\alpha) \\ (\frac{d}{2},1) & (1,1) \end{array}\right]\\
&=\frac{1}{2\pi\textnormal{i}}\frac{d}{dt}\left\{\int_{L}\frac{\Gamma(\frac{d}{2}+z)\Gamma(1+z)}{\Gamma(1+\alpha+\alpha z)}\left(\frac{1}{4}t^{-\alpha}|x|^{2}\right)^{-z}dz\right\}\\
&=\frac{t^{-1}}{2\pi\textnormal{i}}\int_{L}\frac{\Gamma(\frac{d}{2}+z)\Gamma(1+z)}{\Gamma(1+\alpha+\alpha z)}\alpha z \left(\frac{1}{4}t^{-\alpha}|x|^{2}\right)^{-z}dz.
\end{align*}
Using the  relation
$$
\frac{\alpha z}{\Gamma(\alpha+\alpha z+1)}=\frac{1}{\Gamma(\alpha+\alpha z)}-\frac{\alpha}{\Gamma(\alpha+\alpha z+1)},
$$
we obtain
\begin{align*}
&\frac{d}{dt}\textnormal{H}_{12}^{20}\left[\frac{1}{4}t^{-\alpha}|x|^{2}\ \Big|\begin{array}{cc} (1+\alpha,\alpha) \\ (\frac{d}{2},1) & (1,1) \end{array}\right]
\\ &=\frac{t^{-1}}{2\pi\textnormal{i}}\int_{L}\left\{\frac{\Gamma(\frac{d}{2}+z)\Gamma(1+z)}{\Gamma(\alpha+\alpha z)}-\alpha\frac{\Gamma(\frac{d}{2}+z)\Gamma(1+z)}{\Gamma(\alpha+\alpha z+1)}\right\}\left(\frac{1}{4}t^{-\alpha}|x|^{2}\right)^{-z}dz
\\ &=t^{-1}\left\{\textnormal{H}_{12}^{20}\left[\frac{1}{4}t^{-\alpha}|x|^{2}\ \Big|\begin{array}{cc} (\alpha,\alpha) \\ (\frac{d}{2},1) & (1,1) \end{array}\right]-\alpha\textnormal{H}_{12}^{20}\left[\frac{1}{4}t^{-\alpha}|x|^{2}\ \Big|\begin{array}{cc} (1+\alpha,\alpha) \\ (\frac{d}{2},1) & (1,1) \end{array}\right]\right\}.
\end{align*}
Thus,
\begin{align*}
D_{t}^{1-\alpha}p(t,x)&=\frac{d}{dt}I_{t}^{\alpha}p(t,x)\\
&=\frac{d}{dt}\left\{\pi^{-\frac{d}{2}}|x|^{-d}t^{\alpha}\textnormal{H}_{12}^{20}\left[\frac{1}{4}t^{-\alpha}|x|^{2}\ \Big|\begin{array}{cc} (1+\alpha,\alpha) \\ (\frac{d}{2},1) & (1,1) \end{array}\right]
\right\}\\
&=\pi^{-d/2}|x|^{-d}\bigg\{\alpha t^{\alpha-1}\textnormal{H}_{12}^{20}\left[\frac{1}{4}t^{-\alpha}|x|^{2}\ \Big|\begin{array}{cc} (1+\alpha,\alpha) \\ (\frac{d}{2},1) & (1,1) \end{array}\right]\\
&\quad\quad\quad\quad +t^{\alpha} \frac{d}{dt}\textnormal{H}_{12}^{20}\left[\frac{1}{4}t^{-\alpha}|x|^{2}\ \Big|\begin{array}{cc} (1+\alpha,\alpha) \\ (\frac{d}{2},1) & (1,1) \end{array}\right]\bigg\}\\
&=\pi^{-d/2}|x|^{-d}t^{\alpha-1}\textnormal{H}_{12}^{20}\left[\frac{1}{4}t^{-\alpha}|x|^{2}\ \Big|\begin{array}{cc} (\alpha,\alpha) \\ (\frac{d}{2},1) & (1,1) \end{array}\right].
\end{align*}
Similarly, using the relation
$$
\frac{\alpha z}{\Gamma(\alpha z+1)}=\frac{1}{\Gamma(\alpha z)},
$$
we have
\begin{align*}
K(t,x)
&=\frac{d}{dt}\textnormal{H}_{12}^{20}\left[\frac{1}{4}t^{-\alpha}|x|^{2}\ \Big|\begin{array}{cc} (1,\alpha) \\ (\frac{d}{2},1) & (1,1) \end{array}\right] \\
&=\pi^{-\frac{d}{2}}|x|^{-d}t^{-1}\textnormal{H}_{12}^{20}\left[\frac{1}{4}t^{-\alpha}|x|^{2}\ \begin{array}{cc}(0,\alpha) \\ (\frac{d}{2},1) & (1,1)\end{array}\right].
\end{align*}
Therefore \eqref{appendix:q} and \eqref{K equal} are proved.

\subsection{Estimates of $p(t,x)$ and $q(t,x)$}

In this subsection we prove Lemma \ref{p exists}(ii) and Lemma \ref{prop:kernel estimate}.
Since the case $\alpha=1$ is easier, we assume  $\alpha \neq 1$.

By  \eqref{H function} and \eqref{appendix:residue} with $n=0$ and $m=\mu$,
\begin{align*}
&\textnormal{H}_{\nu\mu}^{\mu 0}\left[r\ \Big|\begin{array}{ccc}
(c_{1},\gamma_{1}) & \cdots & (c_{\nu},\gamma_{\nu}) \\
(d_{1},\delta_{1}) & \cdots & (d_{\mu},\delta_{\mu})
\end{array}\right]
=\sum_{i=1}^{m}\sum_{k=0}^{\infty}\textnormal{Res}_{z=\hat{d}_{ik}}
\left[\frac{\prod_{j=1}^{\mu}\Gamma(d_{j}+\delta_{j}z)}{\prod_{j=1}^{\nu}\Gamma(c_{j}+\gamma_{j}z)}r^{-z}\right]\nonumber\\
				&=\sum_{i=1}^{m}\sum_{k=0}^{\infty}\lim_{z\rightarrow \hat{d}_{ik}}\left\{\left(\frac{d}{dz}\right)^{n_{ik}-1}\left(\frac{(z-\hat{d}_{ik})^{n_{ik}}}{(n_{ik}-1)!}\frac{\prod_{j=1}^{\mu}\Gamma(d_{j}+\delta_{j}z)}{\prod_{j=1}^{\nu}\Gamma(c_{j}+\gamma_{j}z)}r^{-z}\right)\right\},
\end{align*}
where $\hat{d}_{ik}=-(d_{i}+k)/\delta_{i}\in P_{1}$ is a pole of the integrand in the contour integral and $n_{ik}$ is its order for $i=1,\ldots,\mu$ and $k=0,1,2,\cdots$.

Let $R:=t^{-\alpha}|x|^{2}$, and denote
$$
\textnormal{H}_{k,l}^{p}(R):=\textnormal{H}_{1+k+l\ 2+k+l}^{2+k+l\ 0}\left[\frac{1}{4}R\ \Big|\begin{array}{ccccc}(1,\alpha) & (0,1) &\cdots &(0,1) \\ (\frac{d}{2},1) & (1,1) & \cdots &(1,1) & (1,1) \end{array}\right]
$$
and
$$
\textnormal{H}_{k,l}^{q}(R):=\textnormal{H}_{1+k+l\ 2+k+l}^{2+k+l\ 0}\left[\frac{1}{4}R\ \Big|\begin{array}{ccccc}(\alpha,\alpha) & (0,1) &\cdots &(0,1) \\ (\frac{d}{2},1) & (1,1) & \cdots &(1,1) & (1,1) \end{array}\right],
$$
where $k,l=0,1,2,\cdots$.
Then by \eqref{appendix:diff},
\begin{align*}
\left|D_{x^{i}}\left(|x|^{-d} \textnormal{H}_{k,0}^{p}(R)\right)\right|&=\left|-dx^{i}|x|^{-d-2}\textnormal{H}_{k,0}^{p}(R)-2x^{i}|x|^{-d-2}\textnormal{H}_{k,1}^{p}(R)\right| \\
	 &=\left|-x^{i}|x|^{-d-2}[d\textnormal{H}_{k,0}^{p}(R)+2\textnormal{H}_{k,1}^{p}(R)]\right|\\
	&\leq N |x|^{-d-1}\left|d\textnormal{H}_{k,0}^{p}(R)+2\textnormal{H}_{k,1}^{p}(R)\right|,
\end{align*}
and
\begin{align*}
&\left|D_{x^{j}}D_{x^{i}}\left(|x|^{-d} \textnormal{H}_{k,0}^{p}(R)\right)\right|\\
&=|D_{x^{j}}\left(-x^{i}|x|^{-d-2}[d\textnormal{H}_{k,0}^{p}(R)+2\textnormal{H}_{k,1}^{p}(R)]\right)|\\
	 &=|(-\delta_{ij}|x|^{-d-2}+(d+2)x^{i}x^{j}|x|^{-d-4})[d\textnormal{H}_{k,0}^{p}(R)+2\textnormal{H}_{k,1}^{p}(R)] \\
	 &\quad\quad\quad+2x^{i}x^{j}|x|^{-d-4}[d\textnormal{H}_{k,1}^{p}(R)+2\textnormal{H}_{k,2}^{p}(R)]|\\
	&\leq N|x|^{-d-2}\sum_{l=1}^{2}|d\textnormal{H}_{k,l-1}^{p}(R)+2\textnormal{H}_{k,l}^{p}(R)|.
\end{align*}
Inductively, for any $m=3,4, \cdots$ and $k=0,1,\cdots$, we have
\begin{align*}
\left|D_{x}^{m}\left(|x|^{-d} \textnormal{H}_{k,0}^{p}(R)\right)\right|\leq N|x|^{-d-m}\sum_{l=1}^{m}\left|d\textnormal{H}_{k,l-1}^{p}(R)+2\textnormal{H}_{k,l}^{p}(R)\right|.
\end{align*}
Hence
\begin{align}
				\label{appendix:p-large}
|\partial_{t}^{n}D_{x}^{m}p(t,x)| &= \left|D_{x}^{m}\left(|x|^{-d}\partial_{t}^{n}\textnormal{H}_{12}^{20}\left[\frac{1}{4}R\ \Big|\begin{array}{cc}(1,\alpha) \\ (\frac{d}{2},1) & (1,1)\end{array}\right]\right)\right|\nonumber \\
& \leq N t^{-n}\sum_{k=1}^{n}\left|D_{x}^{m}\left(|x|^{-d} \textnormal{H}_{k,0}^{p}(R)\right)\right|\nonumber\\
&\leq N|x|^{-d-m}t^{-n}\sum_{k=1}^{n}\sum_{l=1}^{m}\left|d\textnormal{H}_{k,l-1}^{p}(R)+2\textnormal{H}_{k,l}^{p}(R)\right|.
\end{align}
Similarly, 
\begin{align*}
|D_{x}^{m}\left(|x|^{-d} \textnormal{H}_{k,0}^{q}(R)\right)|\leq N|x|^{-d-m}\sum_{l=1}^{m}\left|d\textnormal{H}_{k,l-1}^{q}(R)+2\textnormal{H}_{k,l}^{q}(R)\right|
\end{align*}
and
\begin{align}
				\label{appendix:q-large}
|\partial_{t}^{n}D_{x}^{m}q(t,x)|
& \leq N t^{-n+\alpha-1}\sum_{k=1}^{n}\left|D_{x}^{m}\left(|x|^{-d} \textnormal{H}_{k,0}^{q}(R)\right)\right|\nonumber\\
&\leq N|x|^{-d-m}t^{-n+\alpha-1}\sum_{k=0}^{n}\sum_{l=1}^{m}\left|d\textnormal{H}_{k,l-1}^{q}(R)+2\textnormal{H}_{k,l}^{q}(R)\right|.
\end{align}

Now we prove \eqref{p-1} and \eqref{q-1}.  If $R\geq 1$, by Theorem \ref{appendix:kernel estimate1}
$$
\textnormal{H}_{k,l}^{p}(R)\leq N R^{(\frac{d}{2}+k+l)/(2-\alpha)}\exp\{-\sigma R^{\frac{1}{2-\alpha}}\}
$$
and
$$
\textnormal{H}_{k,l}^{q}(R)\leq NR^{(\frac{d}{2}+k+l+1-\alpha)/(2-\alpha)}\exp\{-\sigma R^{\frac{1}{2-\alpha}}\},
$$
where $\sigma=(2-\alpha)\alpha^{\alpha/(2-\alpha)}$.
Hence by \eqref{appendix:p-large}
\begin{align*}
|\partial_{t}^{n}D_{x}^{m}p(t,x)|&\leq N|x|^{-d-m}t^{-n}\sum_{k=1}^{n}\sum_{l=0}^{m}R^{(\frac{d}{2}+k+l)/(2-\alpha)}\exp\{-\sigma R^{\frac{1}{2-\alpha}}\}\\
&\leq N (t^{-\alpha/2})^{d+m}t^{-n}\exp\{-(\sigma/2) t^{-\frac{\alpha}{2-\alpha}}|x|^{\frac{2}{2-\alpha}}\}\\
&\leq N t^{\frac{-\alpha(d+m)}{2}-n}\exp\{-(\sigma/2) t^{-\frac{\alpha}{2-\alpha}}|x|^{\frac{2}{2-\alpha}}\}.
\end{align*}
Similarly, by \eqref{appendix:q-large}
\begin{align*}
|\partial_{t}^{n}D_{x}^{m}q(t,x)|\leq N t^{\frac{-\alpha(d+m)}{2}-n+\alpha-1}\exp\{-(\sigma/2) t^{-\frac{\alpha}{2-\alpha}}|x|^{\frac{2}{2-\alpha}}\}.
\end{align*}
Thus \eqref{p-1} and \eqref{q-1} are proved.

To prove \eqref{p-3}, we recall \eqref{appendix:gamma1}. For $k,l =0,1,\ldots$,  denote
$$
\Theta_{k,l}^{p}(z)=\frac{\Gamma(\frac{d}{2}+z)\{\Gamma(1+z)\}^{k+l+1}}{\Gamma(1+\alpha z)\{\Gamma(z)\}^{k+l}}=\frac{\Gamma(\frac{d}{2}+z)\Gamma(1+z)}{\Gamma(1+\alpha z)}z^{k+l}.
$$
First we assume that $d$ is an odd number.
Then $\Theta_{k,l}^{p}$ has simple poles at $d_{1j}=-1-j$ and $d_{2j}=-\frac{d}{2}-j$ for $j=0,1,2,\cdots$.
Due to \eqref{appendix:gamma-pole} and \eqref{appendix:residue}, for $R \leq 1$ we have
\begin{align*}
&\textnormal{H}_{k,l}^{p}(R)=\sum_{i=1}^{2}\sum_{j=0}^{\infty}\textnormal{Res}_{z=d_{ij}}\left[\Theta_{k,l}^{p}R^{-z}\right]\nonumber\\
&=\sum_{i=1}^{2}\sum_{j=0}^{\infty}\lim_{z\rightarrow d_{ij}}\left((z-d_{ij})\Theta_{k,l}^{p}(z)R^{-z}\right)\nonumber\\
&=\sum_{j=0}^{\infty}(-1-j)^{k+l}\cdot \frac{(-1)^{j}}{j!}\cdot \frac{\Gamma(\frac{d}{2}-1-j)}{\Gamma(1-\alpha-j\alpha)}R^{1+j}\nonumber \\
&\quad+\sum_{j=0}^{\infty}\left(-\frac{d}{2}-j\right)^{k+l}\cdot\frac{(-1)^{j}}{j!}\cdot\frac{\Gamma(1-\frac{d}{2}-j)}{\Gamma(1-\frac{\alpha d}{2}-\alpha j)}R^{\frac{d}{2}+j}\nonumber\\
&=(-1)^{k+l}\frac{\Gamma(\frac{d}{2}-1)}{\Gamma(1-\alpha)}R+\left(-\frac{d}{2}\right)^{k+l}\frac{\Gamma(1-\frac{d}{2})}{\Gamma(1-\frac{\alpha d}{2})}R^{\frac{d}{2}}+O(R)(R^{\frac{d}{2}+1}+R^2),
\end{align*}
where $O(R)$ is bounded in $(0,1]$ by \eqref{stirling} and \eqref{appendix:gamma1}.
Hence
\begin{align}
                    \label{h0 est 1}
|\textnormal{H}_{k,0}^{p}(R)| \leq N(R+R^{1/2}\cdot 1_{d=1}).
\end{align}
Moreover since
\begin{align}
				\label{cancel}
d\left(-\frac{d}{2}\right)^{k+l-1}+2\left(-\frac{d}{2}\right)^{k+l}=0 \quad \quad l=1,2,\ldots,
\end{align}
it holds that for all $l \geq 1$
\begin{align}
								\label{appendix:d odd}
|d\textnormal{H}_{k,l-1}^{p}(R)+2\textnormal{H}_{k,l}^{p}(R)| \leq NR.
\end{align}

If $d$ is an even number, $\Theta_{k,l}^{p}$ has a simple pole at $d_{1j}=-1-j$ for $0 \leq j \leq \frac{d}{2}-2$ and a pole of order 2 at $d_{2j}=-\frac{d}{2}-j$ for $j=0,1,2,\cdots$.
Hence by \eqref{appendix:gamma-pole} and \eqref{appendix:residue}, for $R \leq 1$ we have
\begin{align*}
&\textnormal{H}_{k,l}^{p}(R)=\sum_{j=0}^{\frac{d}{2}-2}\textnormal{Res}_{z=d_{1j}}\left[\Theta_{k,l}^{p}R^{-z}\right]1_{d\neq 2}+\sum_{j=0}^{\infty}\textnormal{Res}_{z=d_{2j}}\left[\Theta_{k,l}^{p}R^{-z}\right]\nonumber\\
				&=\sum_{j=0}^{\frac{d}{2}-2}\lim_{z\rightarrow d_{1j}}\left((z-d_{1j})\Theta_{k,l}^{p}(z)R^{-z}\right)1_{d\neq 2}+\sum_{j=0}^{\infty}\lim_{z\rightarrow d_{2j}}\frac{d}{dz}\left((z-d_{2j})^{2}\Theta_{k,l}^{p}(z)R^{-z}\right)\nonumber\\
				 &=\sum_{j=0}^{\frac{d}{2}-2}(-1-j)^{k+l}\cdot\frac{(-1)^j}{j!}\cdot\frac{\Gamma(\frac{d}{2}-1-j)}{\Gamma(1-\alpha (1+j))}R^{1+j} \cdot 1_{d\neq 2} \\
&\quad -\sum_{j=0}^{\infty}\left(-\frac{d}{2}-j\right)^{k+l}\cdot\frac{(-1)^{\frac{d}{2}+2j-1}}{j!\left(\frac{d}{2}+j-1\right)!}\cdot\frac{R^{\frac{d}{2}+j}\ln R}{\Gamma(1-\frac{\alpha d}{2}-j\alpha )}\nonumber\\
&\quad +\sum_{j=0}^{\infty}\textnormal{Res}_{z=d_{2j}}\left[\Theta_{k,l}^{p}(z)\right]R^{\frac{d}{2}+j}\nonumber \\
&=(-1)^{k+l}\frac{\Gamma(\frac{d}{2}-1)}{\Gamma(1-\alpha)}R \cdot 1_{d\neq 2}
-\left(-\frac{d}{2}\right)^{k+l}\cdot\frac{(-1)^{\frac{d}{2}-1}}{\left(\frac{d}{2}-1\right)!}\cdot\frac{R^{\frac{d}{2}}\ln R}{\Gamma(1-\frac{\alpha d}{2})}
+ O(R)(R^{\frac{d}{2}}+R^2),
\end{align*}
where $O(R)$ is bounded in $(0,1]$ due to \eqref{stirling} and \eqref{appendix:gamma1}.
Hence
\begin{align}
                    \label{h0 est 2}
|\textnormal{H}_{k,0}^{p}(R)| \leq (R + R \ln R \cdot 1_{d=2}).
\end{align}
Moreover by \eqref{cancel} again, if $l\geq 1$ then
\begin{align}
		\label{appendix:d even}
|d\textnormal{H}_{k,l-1}^{p}(R)+2\textnormal{H}_{k,l}^{p}(R)|\leq NR.
\end{align}
Therefore due to \eqref{appendix:p-large}, \eqref{h0 est 1}, \eqref{appendix:d odd}, \eqref{h0 est 2}, and \eqref{appendix:d even}, we obtain for any $R\leq 1$
\begin{align*}
|\partial_{t}^{n} p(t,x)|
\leq N |x|^{-d}t^{-n} (R + R^{1/2} \cdot 1_{d=1} +R \ln R \cdot 1_{d=2}),
\end{align*}
and for any $R\leq 1$ and $m \in \bN$
\begin{align*}
|\partial_{t}^{n} D_{x}^{m}p(t,x)|
\leq N |x|^{-d-m}t^{-n}R,
\end{align*}
where $N$ depends only on $d$, $m$, $n$, and $\alpha$. Therefore Lemma \ref{p exists}(ii) is proved.

Next we prove \eqref{q-3}. For $k,l=0,1,\ldots$, denote
$$
\Theta_{k,l}^{q}(z)=\frac{\Gamma(\frac{d}{2}+z)\{\Gamma(1+z)\}^{k+l+1}}{\Gamma(\alpha+\alpha z)\{\Gamma(z)\}^{k+l}}=\frac{\Gamma(\frac{d}{2}+z)\Gamma(1+z)}{\Gamma(\alpha(1+z))}z^{k+l}.
$$
We remark that $\Theta_{k,l}^{q}(z)$ does not have a pole at $z=-1$ unless $d=2$.
First  we assume that $d$ is an odd number.
Then $\Theta_{k,l}^{q}$ has a simple pole at $d_{1j}=-2-j$ and $d_{2j}=-\frac{d}{2}-j$ for $j=0,1,2,\cdots$.
By \eqref{appendix:gamma-pole} and \eqref{appendix:residue}, for $R\leq 1$ we have
\begin{align}
				\label{appendix:odd general}
&\textnormal{H}_{k,l}^{q}(R)=\sum_{i=1}^{2}\sum_{j=0}^{\infty}\textnormal{Res}_{z=d_{ij}}\left[\Theta_{k,l}^{q}R^{-z}\right]
				=\sum_{i=1}^{2}\sum_{j=0}^{\infty}\lim_{z\rightarrow d_{ij}}
\left((z-d_{ij})\Theta_{k,l}^{q}(z)R^{-z}\right)\nonumber\\ &=\sum_{j=0}^{\infty}(-2-j)^{k+l}\cdot\frac{(-1)^{1+j}}{(1+j)!}\cdot\frac{\Gamma(\frac{d}{2}-2-j)}{\Gamma(-\alpha(1+j))}R^{2+j}\nonumber \\
&\quad +\sum_{j=0}^{\infty}\left(-\frac{d}{2}-j\right)^{k+l}\cdot\frac{(-1)^{j}}{j!}\cdot\frac{\Gamma(1-\frac{d}{2}-j)}{\Gamma(\alpha-\frac{\alpha d}{2}-j\alpha)}R^{\frac{d}{2}+j}\nonumber\\
				 &=-(-2)^{k+l} \frac{\Gamma(\frac{d}{2}-2)}{\Gamma(-\alpha)} R^2
                +\left(-\frac{d}{2}\right)^{k+l}\frac{\Gamma(1-\frac{d}{2})}{\Gamma(\alpha-\frac{\alpha d}{2})}R^{\frac{d}{2}}+O(R)(R^{\frac{d}{2}+1} + R^3),
\end{align}
where $O(R)$ is bounded in $(0,1]$  as before.
Hence
\begin{align}
                    \label{hq0 est 1}
|\textnormal{H}_{k,0}^{q}(R)| \leq N(R^2 + R^{d/2} \cdot 1_{d=1,3}).
\end{align}
Furthermore, by \eqref{cancel} and \eqref{appendix:odd general} if $l\geq 1$ then
\begin{align}
					\label{appendix:d odd q}
					|d\textnormal{H}_{k,l-1}^{q}(R)+2\textnormal{H}_{k,l}^{q}(R)|\leq N (R^2 + R^{3/2} \cdot 1_{d=1}).
\end{align}
								
On the other hand, if $d$ is even and greater than $2$ then $\Theta_{k,l}^{q}$ has a simple pole at $d_{1j}=-1-j$ for $1\leq j \leq \frac{d}{2}-2$ and a pole of order 2 at $d_{2j}=-\frac{d}{2}-j$ for $j=0,1,2,\cdots$.
If $d=2$, then $\Theta_{k,l}^{q}$ has a simple pole at $-1$ and a pole of order 2 at $d_{2j}=-1-j$ for $j=1,2,\cdots$.
Thus by \eqref{appendix:gamma-pole} and \eqref{appendix:residue}, for $R \leq 1$ we have
\begin{align*}
&\textnormal{H}_{k,l}^{q}(R)
=\sum_{j=1}^{\frac{d}{2}-2}\textnormal{Res}_{z=d_{1j}}\left[\Theta_{k,l}^{q}R^{-z}\right]1_{d \geq 6}
                    \notag
+\textnormal{Res}_{z=-1}\left[\Theta_{k,l}^{q}R^{-z}\right]1_{d=2} \\
                    \notag
&\quad \quad \quad \quad +\textnormal{Res}_{z=d_{20}}\left[\Theta_{k,l}^{q}R^{-z}\right] \cdot 1_{d \geq 4}+\sum_{j=1}^{\infty}\textnormal{Res}_{z=d_{2j}}\left[\Theta_{k,l}^{q}R^{-z}\right]\nonumber\\
&=\sum_{j=1}^{\frac{d}{2}-2}\lim_{z\rightarrow d_{1j}}\left((z-d_{1j})\Theta_{k,l}^{q}(z)R^{-z}\right)1_{d \geq 6}
+\left(-1\right)^{k+l}R \cdot 1_{d=2} \\
&\quad +\lim_{z\rightarrow d_{20}}\frac{d}{dz}\left((z-d_{20})^{2}\Theta_{k,l}^{q}(z)R^{-z}\right)1_{d \geq 4}
+\sum_{j=1}^{\infty}\lim_{z\rightarrow d_{2j}}\frac{d}{dz}\left((z-d_{2j})^{2}\Theta_{k,l}^{q}(z)R^{-z}\right)\nonumber.
\end{align*}
By  the product rule of the differentiation, the above term equals
\begin{align*}
&\sum_{j=1}^{\frac{d}{2}-2} (-1-j)^{k+l}\cdot\frac{(-1)^{j}}{j!}\cdot\frac{\Gamma(\frac{d}{2}-1-j)}{\Gamma(-\alpha j)}R^{1+j}
\cdot 1_{d \geq 6} \\
&\quad +\left(-1\right)^{k+l}R \cdot 1_{d=2}
+\textnormal{Res}_{z=d_{20}}\left[\Theta_{k,l}^{q}(z)\right]R^{\frac{d}{2}}\cdot 1_{d \geq 4} \\
&\quad -\left(-\frac{d}{2}\right)^{k+l}\cdot\frac{(-1)^{\frac{d}{2}-1}}{(\frac{d}{2}-1)!}\cdot\frac{R^{\frac{d}{2}}\ln R}{\Gamma(\alpha-\frac{\alpha d}{2})}1_{d \geq 4}
+\sum_{j=1}^{\infty}\textnormal{Res}_{z=d_{2j}}\left[\Theta_{k,l}^{q}(z)\right]R^{\frac{d}{2}+j}\nonumber\\
&\quad -\sum_{j=1}^{\infty}\left(-\frac{d}{2}-j\right)^{k+l}\cdot\frac{(-1)^{\frac{d}{2}+2j-1}}{j!(\frac{d}{2}+j-1)!}\cdot\frac{R^{\frac{d}{2}+j}\ln R}{\Gamma(\alpha-\frac{\alpha d}{2}-\alpha j)}\nonumber \\
                    \notag
&=-(-2)^{k+l}\cdot\frac{\Gamma(\frac{d}{2}-2)}{\Gamma(- \alpha )}R^2 \cdot 1_{d \geq 6}
+\left(-1\right)^{k+l}R \cdot 1_{d=2}\\
&\quad +\textnormal{Res}_{z=d_{20}}\left[\Theta_{k,l}^{q}(z)\right]R^{\frac{d}{2}} \cdot 1_{d \geq 4}
-\left(-\frac{d}{2}\right)^{k+l}\cdot\frac{(-1)^{\frac{d}{2}-1}}{(\frac{d}{2}-1)!}\cdot\frac{R^{\frac{d}{2}}\ln R}{\Gamma(\alpha-\frac{\alpha d}{2})} 1_{d \geq 4} \\
&\quad +O(R)\big( R^3 + R^{\frac{d}{2}+1}(1+|\ln R|)\big),
\end{align*}
where $O(R)$ is again bounded in $(0,1]$.
Hence
\begin{align}
                    \label{hq0 est 2}
					|\textnormal{H}_{k,0}^{q}(R)|\leq N \big(R^2+R\cdot 1_{d=2}+R^2\ln R \cdot 1_{d=4}\big)
\end{align}
and by \eqref{cancel}
\begin{align}
					\label{appendix:d even q}
|d\textnormal{H}_{k,l-1}^{q}(R)+2\textnormal{H}_{k,l}^{q}(R)|\leq N(R^2 + R^2 \ln R \cdot 1_{d=2}).
\end{align}
Combining
\eqref{appendix:q-large}, \eqref{hq0 est 1}, \eqref{appendix:d odd q}, \eqref{hq0 est 2}, and \eqref{appendix:d even q},
for any $R \leq 1$ and nonnegative integer $m$ we have
\begin{align*}
&|\partial_{t}^{n} D_{x}^{m}q(t,x)|\\
&\leq N|x|^{-d-m}t^{-n+\alpha-1}(R^2+R^2 \ln R \cdot 1_{d=2})\\
& \quad  + N|x|^{-d}t^{-n+\alpha-1}\left[ R^{1/2} \cdot 1_{d=1}+ R \cdot 1_{d=2}+R^2 \ln R \cdot 1_{d=4} \right] 1_{m = 0},
\end{align*}
where $N$ depends only on $d$, $m$, $n$, and $\alpha$.
Therefore Lemma \ref{prop:kernel estimate}(i) is proved.

\vspace{3mm}

Finally we prove Lemma \ref{prop:kernel estimate}(ii), that is for
 any $t\neq 0$ an $x\neq 0$,
\begin{equation}
                    \label{p eqal}
 \partial_{t}^{\alpha}p=\Delta p,
\quad \quad
              \frac{\partial p}{\partial t}=\Delta q.
\end{equation}
To prove the first assertion above one may try to show that their Fourier transforms coincide. But due to the singularity of
$\Delta p(t,\cdot)$ near zero, we instead prove
  \begin{equation}
                    \label{eqn 10.10}
  \int^T_0\int_{\fR^d}\partial^{\alpha}_tp(t,x) h(t)\phi(x)dxdt=\int^T_0\int_{\fR^d}\Delta p(t,x) h(t) \phi(x)dxdt
  \end{equation}
   for any  $\phi \in C_c^\infty(\fR^d\setminus \{0\})$ and $h(t)\in C_c^\infty((0,T))$.

    Note that $t^{-\alpha}|x|^2\geq c>0$ on the support of $h\phi$, and thus  \eqref{p-1} implies that   $\frac{\partial p}{\partial t}(t,x), \partial^{\alpha}_tp(t,x)$, and $\Delta p(t,x)$ are bounded on the support of $h\phi$. Hence both sides of \eqref{eqn 10.10} make sense. By the integration by parts,
    $$
    \int^{T}_0 \partial^{\alpha}_t p(t,x) h dt=\int^{T}_0 p(T-t,x)D^{\alpha}_t H(t)dt,
    $$
    where $H(t):=h(T-t)$.
    Recall that by Parseval's identity
    $$
    \int_{\fR^d} f \bar{g} dx=\int_{\fR^d} (\cF{f}) (\bar{\cF{g}}) d\xi, \quad \forall f,g\in L_2(\fR^d).
    $$
    Considering an approximation of $f$  by functions in $L_2(\fR^d)$ one can easily prove that Parseval's identity holds if $f\in L_1(\fR^d)$ and $g$ is in the Schwartz class.
    Hence the left side of \eqref{eqn 10.10} equals
    \begin{align*}
    &
    \int^{T}_0 \left[\int_{\fR^d} p(T-t,x)\phi(x)dx\right]\partial^{\alpha}_t H(t) dt\\
    &=\int^{T}_0 \left[\int_{\fR^d} E_{\alpha}(-(T-t)^{\alpha}|\xi|^2)\cF(\phi)(\xi)d\xi\right]\partial^{\alpha}_t H(t) dt.
    \end{align*}
    Observe that $\cF(\phi)(\xi) \to 0$ significantly fast as $|\xi|\to \infty$ and $\partial^{\alpha}_t H(t)=0$ if $t$ is sufficiently small.
    Hence by \eqref{mittag}  we can apply Fubini's theorem and show that the last term above is equal to
    \begin{align*}
    &\int_{\fR^d}\left[\int^T_0E_{\alpha}(-(T-t)^{\alpha}|\xi|^2)\partial^{\alpha}_t H(t) dt\right]\cF(\phi)(\xi)d\xi\\
    &= \int_{\fR^d} \int^T_0  E_{\alpha}(-t^{\alpha}|\xi|^2) h(t) \cF(\Delta \phi)(\xi)dtd\xi\\
    &= \int^T_0 \int_{\fR^d}p(t,x) h(t) \Delta \phi(x)dtdx=\int^T_0\int_{\fR^d}\Delta p(t,x) h(t) \phi(x)dxdt.
    \end{align*}
    For the first equality above we use the integration by parts in time variable and the fact $\partial^{\alpha}_t E_{\alpha}(-t^{\alpha}|\xi|^2)=-|\xi|^2E_{\alpha}(-t^{\alpha}|\xi|^2)$, and Parseval's identity is used for the second equality. Therefore the first assertion of \eqref{eqn 10.10}  is proved.

 Next we prove $\Delta q=\frac{\partial p}{\partial t}$.
Note that due to \eqref{p-1}, if $x\neq 0$ then $D_{x}p(\cdot,x)$ and $D_{x}^{2}p(\cdot,x)$ are bounded on  $(0,T)$ uniformly in a neighborhood of $x$. Hence, if  $\alpha\in (0,1)$,
\begin{align*}
\Delta q &=\Delta \frac{d}{dt} \int^t_0 k_{\alpha}(t-s)p(s,x)ds\\
&= \frac{d}{dt}\Delta \int^t_0 k_{\alpha}(t-s)p(s,x)ds\\
&=\frac{d}{dt}\int^t_0 k_{\alpha}(t-s)\Delta p(s,x)ds=D^{1-\alpha}\partial^{\alpha}_tp=\frac{\partial p}{\partial t}.
\end{align*}
Similarly,  if $\alpha\in(1,2)$,
\begin{align*}
\Delta q= \Delta (I_{t}^{\alpha-1}p)=I_{t}^{\alpha-1}(\Delta p)=I_{t}^{\alpha-1}\partial_{t}^{\alpha}p=\frac{\partial p}{\partial t}.
\end{align*}
Thus the second assertion of \eqref{p eqal}  is proved.


\begin{thebibliography}{10}

\bibitem{Bessel}
M.~Abramowitz and I.~A. Stegun.
\newblock {\em Handbook of mathematical functions: with formulas, graphs, and
  mathematical tables}.
\newblock Number~55. Courier Dover Publications, 1972.

\bibitem{Gamma}
G.~E. Andrews, R.~Askey, and R.~Roy.
\newblock {\em Special Functions}, volume~71.
\newblock Cambridge University Press, 1999.

\bibitem{bagley1983theoretical}
R.~L. Bagley and P.~Torvik.
\newblock A theoretical basis for the application of fractional calculus to
  viscoelasticity.
\newblock {\em Journal of Rheology (1978-present)}, 27(3):201--210, 1983.

\bibitem{benson2000application}
D.~A. Benson, S.~W. Wheatcraft, and M.~M. Meerschaert.
\newblock Application of a fractional advection-dispersion equation.
\newblock {\em Water Resources Research}, 36(6):1403--1412, 2000.

\bibitem{braaksma1936asymptotic}
B.~L.~J. Braaksma.
\newblock Asymptotic expansions and analytic continuations for a class of
  barnes-integrals.
\newblock {\em Compositio Mathematica}, 15:239--341, 1936.

\bibitem{caponetto2010fractional}
R.~Caponetto.
\newblock {\em Fractional order systems: modeling and control applications},
  volume~72.
\newblock World Scientific, 2010.

\bibitem{clement2004quasilinear}
P.~Cl{\'e}ment, S.-O. Londen, and G.~Simonett.
\newblock Quasilinear evolutionary equations and continuous interpolation
  spaces.
\newblock {\em Journal of Differential Equations}, 196(2):418--447, 2004.

\bibitem{clement1992global}
P.~Cl{\'e}ment and J.~Pr{\"u}ss.
\newblock Global existence for a semilinear parabolic volterra equation.
\newblock {\em Mathematische Zeitschrift}, 209(1):17--26, 1992.

\bibitem{da1985existence}
G.~Da~Prato and M.~Iannelli.
\newblock Existence and regularity for a class of integrodifferential equations
  of parabolic type.
\newblock {\em Journal of Mathematical Analysis and Applications},
  112(1):36--55, 1985.

\bibitem{Dj}
M.~M. Djrbashian.
\newblock {\em Harmonic analysis and boundary value problems in the complex
  domain}, volume~65.
\newblock Springer, 1993.

\bibitem{EIK}
S.~D. Eidelman, S.~D. Ivasyshen, and A.~N. Kochubei.
\newblock {\em Analytic methods in the theory of differential and
  pseudo-differential equations of parabolic type}, volume 152.
\newblock Springer, 2004.

\bibitem{eidelman2004cauchy}
S.~D. Eidelman and A.~N. Kochubei.
\newblock Cauchy problem for fractional diffusion equations.
\newblock {\em Journal of Differential Equations}, 199(2):211--255, 2004.

\bibitem{engheia1997role}
N.~Engheia.
\newblock On the role of fractional calculus in electromagnetic theory.
\newblock {\em Antennas and Propagation Magazine, IEEE}, 39(4):35--46, 1997.

\bibitem{glockle1995fractional}
W.~G. Gl{\"o}ckle and T.~F. Nonnenmacher.
\newblock A fractional calculus approach to self-similar protein dynamics.
\newblock {\em Biophysical Journal}, 68(1):46--53, 1995.

\bibitem{hilfer2000applications}
R.~Hilfer, P.~Butzer, U.~Westphal, J.~Douglas, W.~Schneider, G.~Zaslavsky,
  T.~Nonnemacher, A.~Blumen, and B.~West.
\newblock {\em Applications of fractional calculus in physics}, volume~5.
\newblock World Scientific, 2000.

\bibitem{kilbas2004h}
A.~A. Kilbas.
\newblock {\em H-transforms: Theory and Applications}.
\newblock CRC Press, 2004.

\bibitem{Kim2014BMOpseudo}
I.~Kim, K.-H. Kim, and S.~Lim.
\newblock Parabolic {BMO} estimates for pseudo-differential operators of
  arbitrary order.
\newblock {\em arXiv preprint arXiv:1408.2343}, 2014.

\bibitem{kochubei2014asymptotic}
A.~N. Kochubei.
\newblock Asymptotic properties of solutions of the fractional diffusion-wave
  equation.
\newblock {\em Fractional Calculus and Applied Analysis}, 17(3):881--896, 2014.

\bibitem{Krylov1999}
N.~V. Krylov.
\newblock An analytic approach to {SPDEs}.
\newblock {\em Stochastic Partial Differential Equations: Six Perspectives,
  Mathematical Surveys and Monographs}, 64:185--242, 1999.

\bibitem{krylov2001calderon}
N.~V. Krylov.
\newblock On the {Calder{\'o}n-Zygmund} theorem with applications to parabolic
  equations.
\newblock {\em Algebra i Analiz}, 13(4):1--25, 2001.

\bibitem{kunstmann2004maximal}
P.~C. Kunstmann and L.~Weis.
\newblock {Maximal $L_{p}$-regularity for Parabolic Equations, Fourier
  Multiplier Theorems and $H^{\infty}$-functional Calculus}.
\newblock In {\em Functional analytic methods for evolution equations}, pages
  65--311. Springer, 2004.

\bibitem{langlands2009fractional}
T.~Langlands, B.~Henry, and S.~Wearne.
\newblock Fractional cable equation models for anomalous electrodiffusion in
  nerve cells: infinite domain solutions.
\newblock {\em Journal of mathematical biology}, 59(6):761--808, 2009.

\bibitem{mainardi1995fractional}
F.~Mainardi.
\newblock Fractional diffusive waves in viscoelastic solids.
\newblock {\em Nonlinear Waves in Solids, Fairfield}, pages 93--97, 1995.

\bibitem{MS}
M.~M. Meerschaert and A.~Sikorskii.
\newblock {\em Stochastic models for fractional calculus}, volume~43.
\newblock Walter de Gruyter, 2011.

\bibitem{metzler1999anomalous}
R.~Metzler, E.~Barkai, and J.~Klafter.
\newblock Anomalous diffusion and relaxation close to thermal equilibrium: a
  fractional fokker-planck equation approach.
\newblock {\em Physical Review Letters}, 82(18):3563, 1999.

\bibitem{metzler2000boundary}
R.~Metzler and J.~Klafter.
\newblock Boundary value problems for fractional diffusion equations.
\newblock {\em Physica A: Statistical Mechanics and its Applications},
  278(1):107--125, 2000.

\bibitem{MK}
R.~Metzler and J.~Klafter.
\newblock The random walk's guide to anomalous diffusion: a fractional dynamics
  approach.
\newblock {\em Physics reports}, 339(1):1--77, 2000.

\bibitem{metzler2004restaurant}
R.~Metzler and J.~Klafter.
\newblock The restaurant at the end of the random walk: recent developments in
  the description of anomalous transport by fractional dynamics.
\newblock {\em Journal of Physics A: Mathematical and General}, 37(31):R161,
  2004.

\bibitem{metzler1995relaxation}
R.~Metzler, W.~Schick, H.-G. Kilian, and T.~F. Nonnenmacher.
\newblock Relaxation in filled polymers: A fractional calculus approach.
\newblock {\em The Journal of Chemical Physics}, 103(16):7180--7186, 1995.

\bibitem{ortigueira2011fractional}
M.~D. Ortigueira.
\newblock {\em Fractional calculus for scientists and engineers}, volume~84.
\newblock Springer, 2011.

\bibitem{Po}
I.~Podlubny.
\newblock {\em Fractional differential equations: an introduction to fractional
  derivatives, fractional differential equations, to methods of their solution
  and some of their applications}, volume 198.
\newblock Academic press, 1998.

\bibitem{podlubny1999fractional}
I.~Podlubny.
\newblock {Fractional-order systems and $PI^{\lambda}D^{\mu}$-controllers}.
\newblock {\em IEEE Transactions on Automatic Control}, 44(1):208--214, 1999.

\bibitem{Pr1991}
J.~Pr{\"u}ss.
\newblock Quasilinear parabolic {Volterra} equations in spaces of integrable
  functions.
\newblock {\em In B. de Pagter, Ph. Cl\'{e}ment, E. Mitidieri, editors,
  Semigroup Theory and Evolution Equations, Lecture Notes in Pure and Applied
  Mathematics}, 135:401--420, 1991.

\bibitem{Pr}
J.~Pr{\"u}ss.
\newblock {\em Evolutionary integral equations and applications}.
\newblock Springer, 2012.

\bibitem{pskhu2009fundamental}
A.~V. Pskhu.
\newblock The fundamental solution of a diffusion-wave equation of fractional
  order.
\newblock {\em Izvestiya: Mathematics}, 73(2):351, 2009.

\bibitem{raberto2002waiting}
M.~Raberto, E.~Scalas, and F.~Mainardi.
\newblock Waiting-times and returns in high-frequency financial data: an
  empirical study.
\newblock {\em Physica A: Statistical Mechanics and its Applications},
  314(1):749--755, 2002.

\bibitem{sabatier2007advances}
J.~Sabatier, O.~P. Agrawal, and J.~T. Machado.
\newblock {\em Advances in fractional calculus}.
\newblock Springer, 2007.

\bibitem{SY}
K.~Sakamoto and M.~Yamamoto.
\newblock Initial value/boundary value problems for fractional diffusion-wave
  equations and applications to some inverse problems.
\newblock {\em Journal of Mathematical Analysis and Applications},
  382(1):426--447, 2011.

\bibitem{SKM}
S.~G. Samko, A.~A. Kilbas, and O.~I. Marichev.
\newblock {\em Fractional integrals and derivatives}, volume 1993.
\newblock 1993.

\bibitem{scalas2000fractional}
E.~Scalas, R.~Gorenflo, and F.~Mainardi.
\newblock Fractional calculus and continuous-time finance.
\newblock {\em Physica A: Statistical Mechanics and its Applications},
  284(1):376--384, 2000.

\bibitem{SBMW}
R.~Schumer, D.~A. Benson, M.~M. Meerschaert, and S.~W. Wheatcraft.
\newblock Eulerian derivation of the fractional advection--dispersion equation.
\newblock {\em Journal of Contaminant Hydrology}, 48(1):69--88, 2001.

\bibitem{stein1971introduction}
E.~M. Stein and G.~L. Weiss.
\newblock {\em Introduction to Fourier analysis on Euclidean spaces}, volume~1.
\newblock Princeton university press, 1971.

\bibitem{tarasov2006electromagnetic}
V.~E. Tarasov.
\newblock Electromagnetic fields on fractals.
\newblock {\em Modern Physics Letters A}, 21(20):1587--1600, 2006.

\bibitem{V}
L.~von Wolfersdorf.
\newblock On identification of memory kernels in linear theory of heat
  conduction.
\newblock {\em Mathematical methods in the applied sciences}, 17(12):919--932,
  1994.

\bibitem{YGD}
H.~Ye, J.~Gao, and Y.~Ding.
\newblock A generalized {Gronwall} inequality and its application to a
  fractional differential equation.
\newblock {\em Journal of Mathematical Analysis and Applications},
  328(2):1075--1081, 2007.

\bibitem{zacher2005maximal}
R.~Zacher.
\newblock {Maximal regularity of type $L_{p}$ for abstract parabolic Volterra
  equations}.
\newblock {\em Journal of Evolution Equations}, 5(1):79--103, 2005.

\bibitem{Za}
R.~Zacher.
\newblock Weak solutions of abstract evolutionary integro-differential
  equations in hilbert spaces.
\newblock {\em Funkcialaj Ekvacioj}, 52(1):1--18, 2009.

\bibitem{zaslavsky2002chaos}
G.~M. Zaslavsky.
\newblock Chaos, fractional kinetics, and anomalous transport.
\newblock {\em Physics Reports}, 371(6):461--580, 2002.

\end{thebibliography}

\end{document}